\RequirePackage{fix-cm}
\documentclass{amsart} 
\usepackage[utf8]{inputenc}
\usepackage{mathptmx}
\usepackage{latexsym}
\usepackage{amsmath}
\usepackage{amssymb}
\usepackage{amsfonts} 
\usepackage{eucal} 
\usepackage{amsbsy}
\usepackage{amsthm}
\usepackage[all]{xy}
\usepackage{hyperref}
\usepackage{graphicx}

\newtheorem{thm}{Theorem}[section]
\newtheorem*{thm*}{Theorem}
\newtheorem{lemma}[thm]{Lemma}
\newtheorem{prop}[thm]{Proposition}

\newtheorem*{cor*}{Corollary}

\theoremstyle{definition}
\newtheorem{defn}[thm]{Definition}
\newtheorem{example}[thm]{Example}

\theoremstyle{remark}

\newcommand {\mbl}   {\ensuremath{\mathbb{L}}}

\newcommand {\real}  {\ensuremath{\mathbb{R}}}
\newcommand {\intg}  {\ensuremath{\mathbb{Z}}}

\newcommand {\rat}   {\ensuremath{\mathbb{Q}}}

\newcommand {\Smash} {\ensuremath{\wedge}}

\newcommand {\pro}   {\ensuremath{\operatorname{pr}}}

\newcommand {\redh}  {\ensuremath{\widetilde{H}}}

\newcommand {\syml}  {\ensuremath{\mathbb{L}^{\bullet}}}
\newcommand {\redsyml}  {\ensuremath{\widetilde{\mathbb{L}^\bullet}}}

\newcommand {\cone}  {\ensuremath{\operatorname{cone}}}

\newcommand {\id}    {\ensuremath{\operatorname{id}}}

\newcommand {\BSPL}   {\ensuremath{\operatorname{BSPL}}}

\newcommand {\BBSPL}   {\ensuremath{\operatorname{B}\widetilde{\operatorname{SPL}}}}

\newcommand {\BSTOP}   {\ensuremath{\operatorname{BSTOP}}}

\newcommand {\MSO}   {\ensuremath{\operatorname{MSO}}}

\newcommand {\MSPL}   {\ensuremath{\operatorname{MSPL}}}

\newcommand {\RMSPL}   {\ensuremath{\widetilde{\operatorname{MSPL}}}}
\newcommand {\MSTOP}   {\ensuremath{\operatorname{MSTOP}}}

\newcommand {\RMSTOP}   {\ensuremath{\widetilde{\operatorname{MSTOP}}}}

\newcommand {\SPL}   {{\ensuremath{\operatorname{SPL}}}}
\newcommand {\PL}   {{\ensuremath{\operatorname{PL}}}}

\newcommand {\STOP}   {{\ensuremath{\operatorname{STOP}}}}

\newcommand {\TOP}   {{\ensuremath{\operatorname{TOP}}}}

\newcommand {\PLB}   {{\ensuremath{\operatorname{PLB}}}}

\newcommand {\Witt}   {{\ensuremath{\operatorname{Witt}}}}

\newcommand {\MWITT}   {\ensuremath{\operatorname{MWITT}}}

\newcommand {\Th}   {\ensuremath{\operatorname{Th}}}

\newcommand {\PB}   {{\ensuremath{\operatorname{PB}}}}

\begin{document}

%******************** TITLE PAGE *****************

\title[Bundle Transfer of L-Orientations for Singular Spaces]
  {Bundle Transfer of L-Homology Orientation Classes for Singular Spaces}

\author{Markus Banagl}

\address{Mathematisches Institut, Universit\"at Heidelberg,
  Im Neuenheimer Feld 205, 69120 Heidelberg, Germany}

\email{banagl@mathi.uni-heidelberg.de}

\thanks{This work is funded by the
 Deutsche Forschungsgemeinschaft (DFG, German Research Foundation)
 -- research grant 495696766.}

\date{April, 2022}

\subjclass[2010]{55N33, 55R12, 57Q20, 57Q50, 57R20, 57N80}
% 55N33   	Intersection homology and cohomology 
% 55R12   	Transfer 
% 57Q20   	Cobordism (in PL topology)
% 57Q50   	Microbundles and block bundles (in PL topology) 
% 57R20   	Characteristic classes and numbers (manifolds and cell complexes) 
% 57N80   	Stratifications in topological manifolds

\keywords{Intersection Homology, Stratified Spaces, Transfer Homomorphisms, 
 Characteristic Classes, Algebraic L-Theory, Block Bundles, Cobordism}

%********************** ABSTRACT ****************************

\begin{abstract}
We consider transfer maps on ordinary homology, bordism of
singular spaces and homology with coefficients in Ranicki's
symmetric L-spectrum, associated to
block bundles with closed oriented PL manifold fiber
and compact polyhedral base. 
We prove that if the base polyhedron
is a Witt space, for example a pure-dimensional compact complex algebraic
variety, then the symmetric L-homology orientation of the base,
constructed by Laures, McClure and the author, transfers
to the L-homology orientation of the total space.
We deduce from this that
the Cheeger-Goresky-MacPherson L-class of the
base transfers to the product of the L-class of the total space
with the cohomological L-class of the stable vertical
normal microbundle.  
\end{abstract}

\maketitle

%******************** TABLE OF CONTENTS *********************

\tableofcontents

%=================================================================
%=================================================================
%=================================================================

\section{Introduction}

To a fiber bundle $p:X\to B$ whose structure group
is a compact Lie group acting smoothly on the compact smooth 
$d$-dimensional manifold fiber $F$,
and whose base space $B$ is a finite complex,
Becker and Gottlieb associate in \cite{beckergottlieb}
a transfer homomorphism
$p^!: H_n (B) \to H_{n+d} (X)$.
Boardman discusses this transfer and several closely related
constructions such as the Umkehr map and pullback transfers
in \cite{boardman}.
Let $L^* (\alpha)$ denote the cohomological Hirzebruch $L$-class of a 
vector bundle $\alpha$, and for a smooth closed oriented manifold $M$
with tangent bundle $TM$, let $L_* (M) \in H_* (M;\rat)$
denote the Poincar\'e dual of $L^* (TM)$. Suppose that $F$ is oriented
and the structure group of $p$ acts in an orientation preserving manner.
If the base $B$ of the fiber bundle is a smooth closed oriented 
manifold $M$, then 
\begin{equation} \label{equ.transferoflmformfds}
p^! L_* (M) = L^* (T_p)^{-1} \cap L_* (X), 
\end{equation}
where $T_p$ is the vertical tangent bundle of $p$.
This is a straightforward consequence of the bundle isomorphism
$TX \cong p^* TM \oplus T_p$, naturality and the Whitney sum formula
for the cohomological $L$-class, multiplicative properties of the
transfer, and the fact that $p^!$ maps the fundamental class of the base
to the fundamental class of the total space.\\

If the base $B$ is a singular pseudomanifold, then the above argument
does not apply. On the other hand, intersection homology methods still allow
for the construction of a homological $L$-class
$L_* (B) \in H_* (B;\rat)$ for many types of compact pseudomanifolds $B$:
In the case where $B$ allows for a stratification with only
even-codimensional strata, for example a pure-dimensional compact complex algebraic
variety, $L_* (B)$ has been defined by Goresky and MacPherson in
\cite{gmih1}. This construction has been extended by Siegel \cite{siegel}
to Witt spaces, i.e.
oriented polyhedral pseudomanifolds that may have strata of odd codimension
such that the middle-dimensional middle-perversity rational intersection homology 
of the corresponding links vanishes.
In \cite{banagl-mem}, \cite{banagl-lcl}, 
the author has yet more generally defined $L_* (B)$
for topologically stratified spaces $B$ 
that allow for Lagrangian structures along strata of
odd codimension. A local definition of $L$-classes for triangulated pseudomanifolds
with piecewise flat metric was given by Cheeger \cite{cheeger} in terms of
$\eta$-invariants of links. As for manifolds, the $L$-class of singular spaces plays
an important role in the topological classification of such spaces, as shown
by Cappell and Weinberger in \cite{cw2} and by Weinberger in \cite{weinberger}.\\

Let $F$ be a closed oriented $d$-dimensional PL manifold, $B$ a compact polyhedron
and $\xi$ an oriented PL $F$-block bundle over $B$ (Casson \cite{casson}).
Oriented PL $F$-fiber bundles $p: X\to B$ are a special case. Block bundles,
and hence the results of this paper, do not require a 
locally trivial projection map $p$.
Then $\xi$ still admits a transfer homomorphism
\[ \xi^!: H_n (B) \to H_{n+d} (X), \]
see Ebert and Randal-Williams
\cite{ebertrandal}, and Section \ref{sec.bundletransferstable} of this paper.
Furthermore, $\xi$ possesses a stable vertical normal
PL microbundle $\nu_\xi$, see
Hebestreit, Land, L\"uck, Randal-Williams \cite{hllrw}
and Section \ref{sec.stableverticalblockbundles} of this paper.
In the present paper we develop methods that yield, 
among other results, formula (\ref{equ.transferoflmformfds})
for $F$-block bundles over Witt spaces $B$
(Theorem \ref{thm.lclassbundletransfer}):

\begin{thm*}
Let $B$ be a closed Witt space (e.g. a pure-dimensional
compact complex algebraic variety)
and let $F$ be a closed oriented PL manifold.
Let $\xi$ be an oriented PL $F$-block bundle over $B$
with total space $X$ and
oriented stable vertical normal microbundle $\nu_\xi$ over $X$.
Then $X$ is a Witt space and 
the associated block bundle transfer $\xi^!$ sends the
Cheeger-Goresky-MacPherson-Siegel $L$-class of $B$ to the product
\begin{equation} \label{equ.mainbundletransferformulaintro} 
\xi^! L_* (B) = L^* (\nu_\xi) \cap L_* (X). 
\end{equation}
\end{thm*}
Note that since the cohomological class $L^* (\nu_\xi)$ is invertible,
this formula yields a method for computing the Cheeger-Goresky-MacPherson
$L$-class of the total space.

Our method of proof rests on the geometric description of 
PL cobordism provided by Buoncristiano, Rourke and Sanderson
in \cite{buonrs} in terms of mock bundles. We construct a transfer
$\xi^!: E_n (B) \to E_{n+d} (X)$ for any module spectrum
$E$ over the Thom spectrum $\MSPL$ of oriented PL-bundle theory.
In addition to the transfer on ordinary homology,
this yields transfer homomorphisms on Ranicki's homology with coefficients
in the symmetric $\syml$-spectrum and on Witt bordism theory, $\Omega^\Witt_*$.
We describe the latter transfer geometrically as a pullback transfer
and use this together with mock bundle theory to show that the 
Witt bordism transfer sends
the fundamental class $[B]_\Witt \in \Omega^\Witt_* (B)$ to the
fundamental class $[X]_\Witt \in \Omega^\Witt_* (X)$
(Proposition \ref{prop.transferofwittfundclass}).
Using the work of Laures, McClure and the author (\cite{blm}),
which provides a map of ring spectra
$\MWITT \to \syml (\rat)$, where $\MWITT$ represents Witt-bordism,
as well as a fundamental class $[B]_\mbl \in \syml (\rat)_* (B)$,
we then show (Theorem \ref{prop.gysinpreserveslhomfundclass}):

\begin{thm*} 
Let $B$ be a closed Witt space of dimension $n$
and let $F$ be a closed oriented PL manifold of dimension $d$.
Let $\xi$ be an oriented PL $F$-block bundle over $B$
with total space $X$.
Then the $\syml$-homology block bundle transfer 
\[ \xi^!:  \syml (\rat)_n (B) \longrightarrow
         \syml (\rat)_{n+d} (X) \]
maps the $\syml (\rat)$-homology fundamental class of $B$ to the
$\syml (\rat)$-homology fundamental class of $X$,
\[ \xi^! [B]_\mathbb{L} = [X]_\mathbb{L}. \]         
\end{thm*}
The result on Cheeger-Goresky-MacPherson $L$-classes is then deduced from an explicit formula for
the transfer by tensoring with the rationals.
For a PL $F$-fiber bundle $p:X\to B$ over a PL manifold base $B$,
the formula 
\[ p^! [B]_\mathbb{L} = [X]_\mathbb{L} \in \syml (\intg)_{n+d} (X) \]
was stated by L\"uck and Ranicki in \cite{lueckranicki}.
The behavior of the $L$-class for singular spaces under
transfers associated to finite degree covering projections has already been
clarified in \cite{banaglcovertransfer}, where we showed that for a
closed oriented Whitney stratified pseudomanifold $B$ admitting Lagrangian structures
along strata of odd codimension (e.g. $B$ Witt)
and $p: X \to B$ an orientation preserving covering map of finite degree,
the $L$-class of $B$ transfers to the $L$-class of the cover, i.e.
\[ p^! L_* (B) = L_* (X). \]
For the Witt case, this is from the perspective of the present paper
a special case of (\ref{equ.mainbundletransferformulaintro}).

An inclusion $g: Y \hookrightarrow X$ of stratified spaces is called
normally nonsingular if $Y$ possesses a tubular neighborhood in $X$
that can be equipped with the structure of a real vector bundle;
see e.g. \cite{gmsmt}, \cite{banaglnyjm}.
An oriented normally nonsingular inclusion $g$ 
of real codimension $c$ has a Gysin map 
\[ g^!: H_* (X;\rat) \longrightarrow H_{*-c} (Y;\rat) \]
on ordinary homology, 
\[ g^!: \syml (\rat)_* (X) \longrightarrow \syml_{*-c} (\rat) (Y) \]
on $\syml (\rat)$-homology, and
\[ g^!: \Omega^\Witt_* (X) \longrightarrow \Omega^\Witt_{*-c} (Y) \]
on Witt bordism.
In \cite{banaglnyjm}, we showed that if $g$ is
a normally nonsingular inclusion of closed
oriented even-dimensional piecewise-linear Witt pseudomanifolds,
for example pure-dimensional compact complex algebraic varieties, then
\begin{equation} \label{equ.gysinlformula}
g^! L_* (X) = L^* (\nu_g) \cap L_* (Y), 
\end{equation}
\[   g^! [X]_\mathbb{L} = [Y]_\mathbb{L},~
   g^! [X]_\Witt = [Y]_\Witt \]
where $\nu_g$ is the normal bundle of $g$. These formulae have been
applied in \cite{banaglnyjm} to compute the Cheeger-Goresky-MacPherson
$L$-class of certain singular Schubert varieties. No previous computations
of such classes seem to be available in the literature.
Together with the bundle transfer formula 
(\ref{equ.mainbundletransferformulaintro}), this makes it possible
to compute the transfer of the Cheeger-Goresky-MacPherson $L$-class
associated to a normally nonsingular
map, i.e. a map which can be
factored as a composition of a normally nonsingular inclusion,
followed by the projection of an oriented PL $F$-fiber bundle $\xi$
with closed PL manifold fiber $F$, see Section 
\ref{sec.normnonsingmaps}.

For complex algebraic, possibly singular, varieties $X$, Brasselet, Sch\"urmann and Yokura
introduced in \cite{bsy} Hodge-theoretic intersection Hirzebruch characteristic classes
$IT_{y,*} (X)$, which agree with $L_* (X)$ for $y=1$ and $X$ nonsingular
or, more generally, a rational homology manifold \cite{bobadilla}.
Using results of Sch\"urmann \cite{schuermannspecial}
and Maxim and Sch\"urmann \cite{msinschubertvar},
we established an algebraic version of (\ref{equ.gysinlformula})
for $IT_{1,*}$
in a context of appropriately normally nonsingular regular algebraic embeddings
\cite[Theorem 6.30]{banaglnyjm}.
Similarly, we expect $IT_{1*}$ to satisfy a relation analogous to
(\ref{equ.mainbundletransferformulaintro}) for
smooth algebraic morphisms $p:X\to B$, where $\nu_p$ would now be inverse to
the algebraic relative tangent bundle $T_{X/B}$. Such a relation together
with the results of the present paper then enable further comparison
between the Hodge-theoretic class $IT_{1*}$ and the topological class $L_*$.
The aforementioned normally nonsingular maps form a topological parallel
to the algebraic concept of a local complete intersection morphism, i.e.
a morphism of varieties that can be factored into a closed regular embedding
and a smooth morphism. Hence the results of the present paper impact the 
behavior of topological characteristic classes under transfers associated
to local complete intersection morphisms.

\section{Stable Vertical Normal Block Bundles}
\label{sec.stableverticalblockbundles}

Block bundles with manifold fiber over compact polyhedra 
admit stable vertical normal closed disc block bundles,
see e.g. \cite{ebertrandal}, \cite{hllrw}, 
as well as \cite[p. 83]{buonrs} for 
the more general mock bundle situation.
We will use the vertical normal block bundle later in the
description of the Umkehr map and thus recall the construction
in the form we need it for our purposes.

Let $F$ be a closed oriented PL manifold of dimension $d$ and
let $K$ be a finite ball complex with associated
polyhedron $B=|K|$. (The polyhedron $B$ is not required to
be a manifold.) 
Let $\xi$ be an oriented PL $F$-block bundle over $K$ (Casson \cite{casson})
with total space $X=E(\xi)$.
Let $b$ denote the dimension of $B$ so that $\dim X = d+b.$
The block of $\xi$ over a cell $\sigma \in K$ will be denoted
by $\xi (\sigma)$. For every $\sigma$, there is a block-preserving
PL homeomorphism
$\xi (\sigma) \cong F \times \sigma$. Thus the blocks of $\xi$
are compact PL manifolds with boundary 
\[ \xi (\partial \sigma) := \bigcup_{\tau \in \partial \sigma} \xi (\tau). \]
Over the interiors 
$\overset{\circ}{\sigma}$ of cells, we set
$\xi (\overset{\circ}{\sigma}) := \xi (\sigma) - \xi (\partial \sigma)$.

In order to construct a stable vertical normal PL block bundle
(and hence a stable PL microbundle, since $\BSPL \simeq \BBSPL$)
for $\xi$, choose a block preserving PL embedding
\[ \theta: X \hookrightarrow \real^s \times B\]
for sufficiently large $s>2d+b+1$, i.e. a PL embedding such that
\[ \theta (\xi (\overset{\circ}{\sigma})) 
  \subset \real^s \times \overset{\circ}{\sigma} \]
and
\[ \theta|: (\xi (\sigma), \xi (\partial \sigma)) 
  \longrightarrow (\real^s \times \sigma, 
    \real^s \times \partial \sigma) \]
is a locally flat PL embedding of manifolds
for every simplex $\sigma \subset K$.
One way to obtain such an embedding is to
choose first a PL embedding $e: X \hookrightarrow \real^s$.
By Casson \cite[Lemma 6, p. 43]{casson}, $\xi$ can be equipped
with a choice of block fibration $p: X\to B$. This is a PL map
such that $\xi (\sigma) = p^{-1} (\sigma)$ for every cell 
$\sigma \in K$.
Then
$\theta := (e,p): X \hookrightarrow \real^s \times B$
is a block preserving PL embedding. (The local flatness is ensured 
by requiring the codimension to be at least $3$.)
Another method to construct a block preserving embedding 
$\theta$ is by induction
over the cells $\sigma \in K$, starting with the $0$-cells $\sigma^0$
and embeddings 
$\theta: \xi (\sigma^0)\cong F \subset \real^s \times \sigma^0 \cong \real^s$.
These are then extended to proper embeddings of manifolds-with-boundary
$\theta: \xi (\sigma^1) \subset \real^s \times \sigma^1$
for every $1$-cell $\sigma^1$ in $K$, etc. 
As in \cite{rosablockbundles1}, 
an embedding $j:M\to Q$ of manifolds is \emph{proper}
if $j^{-1}(\partial Q)=\partial M$.

Recall that one says that a PL embedding $j: A\to P$ of polyhedra
possesses a 
\emph{normal PL closed disc block bundle} if there exists a regular
neighborhood $N$ of $j(A)$ in $P$ such that $N$ is the total space
of a PL closed disc block bundle over $j(A)$ whose zero section
embedding agrees with the inclusion $j(A) \subset N$.

\begin{prop} \label{thm.existencenormaldiscblockbundle}
Let $\xi$ be an $F$-block bundle over a finite cell complex $K$
with polyhedron $B=|K|$,
where $F$ is a closed PL manifold.
A block preserving PL embedding $\theta: X \to \real^s \times B$ 
of the total space $X$ of $\xi$ possesses
a normal PL closed $(s-d)$-disc block bundle $\nu_\theta$ over 
$\theta (X)$.
If $\xi$ is oriented, then $\nu_\theta$ is canonically oriented.
\end{prop}
\begin{proof}
We begin by constructing
a particular regular neighborhood $N$ of 
$\theta (X)$ in $\real^s \times B$ such that $N$ is compatible
with the blocks $\theta (\xi (\sigma))$ and $\real^s \times \sigma$
for all cells $\sigma \in K$.
There exists a locally finite simplicial complex
$L$ with subcomplexes $T,L_\sigma \subset L$ 
($\sigma \in K$) such that the following
properties hold:\\

\noindent (i) $|L| = \real^s \times B$,\\

\noindent (ii) $\theta (X) = |T|$,\\

\noindent (iii) for every cell $\sigma$ in $K$, 
        $\real^s \times \sigma = |L_\sigma|,$ and\\

\noindent (iv) each simplex of $L$ meets $T$ in a single face or not at all.\\

\noindent It follows from (ii) and 
$\theta (\xi (\sigma))\subset \real^s \times \sigma$
that the compact manifold
$M_\sigma := \theta (\xi (\sigma))$ is triangulated by $L_\sigma \cap T$.
The boundary of $M_\sigma$ is triangulated
by the subcomplex $L_{\partial \sigma} \cap T$ of $L$, where
$L_{\partial \sigma}$ is the subcomplex of $L$
given by
\[ L_{\partial \sigma} = \bigcup_{\tau \in \partial \sigma}
     L_\tau. \]
Furthermore, (iv) implies that
each simplex of $L_\sigma$ meets $L_\sigma \cap T$ 
in a single face or not at all.
Let $f: L \to [0,1] = \Delta^1$ be the unique simplicial map
such that $f^{-1} (0) = |T|$.
Then the preimage
\[ N := f^{-1} [0, 1/2] \subset \real^s \times B \]
is a regular neighborhood of $\theta (X)$ in $\real^s \times B$.
The intersection $Q_\sigma := N \cap (\real^s \times \sigma)$ 
is a regular neighborhood of the manifold $M_\sigma$ 
in the manifold $\real^s \times \sigma$.
This regular neighborhood
meets the boundary $\real^s \times \partial \sigma$
transversely, i.e.
$N\cap (\real^s \times \partial \sigma)$ is a regular 
neighborhood of
$\theta (\xi (\sigma)) \cap (\real^s \times \partial \sigma)
= \theta (\xi (\partial \sigma))$
in $\real^s \times \partial \sigma$.
The boundary of the compact manifold $Q_\sigma$ is described by
\begin{equation} \label{equ.boundaryqsigma}
\partial Q_\sigma =
(f^{-1} (1/2) \cap (\real^s \times \sigma)) \cup
 \bigcup_{\tau \in \partial \sigma} Q_\tau 
\end{equation}
and $M_\sigma$ is a proper submanifold of $Q_\sigma$.

We will construct a PL closed disc block bundle
$\nu_\theta$ over $\theta (X)$ by induction on the cells $\sigma$ of $K$. 
The total space $E(\nu_\theta)$ of $\nu_\theta$
is given by $E(\nu_\theta) := N$.
Given a nonnegative integer $n$, we set
\[ L_n := \bigcup_\sigma L_\sigma, \]
where the union is taken over all cells $\sigma \in K$
of dimension at most $n$. The corresponding polyhedron is 
$|L_n| = \real^s \times B^n$, where
$B^n$ denotes the $n$-skeleton of $B$.
Set
\[ Q_n := \bigcup_\sigma Q_\sigma \subset \real^s \times B^n, \]
where the union is taken over all cells $\sigma \in K$
of dimension at most $n$, so that
\[ Q_n = N\cap (\real^s \times B^n). \]
For $n=b=\dim B$, $B^n = B$
and thus $Q_b = N$.

Let $\sigma$ be a $0$-cell of $K$.
By \cite[Thm. 4.3, p. 16]{rosablockbundles1},
there is a disc block bundle
$\nu_\sigma$ over the complex $L_\sigma \cap T$ with total space
$E(\nu_\sigma) = Q_\sigma.$
Then the collection of blocks
$\nu_\sigma (\beta),$ $\beta \in L_\sigma \cap T,$
of the bundles $\nu_\sigma$ endow $Q_0$ with the structure of
a disc block bundle $\nu_0$ over $L_0 \cap T$.
Assume inductively that a block bundle $\nu_{n-1}$
over $L_{n-1} \cap T$ with total space
\[ E(\nu_{n-1}) = Q_{n-1} \]
has been constructed such that
for all cells $\sigma \in K$ with 
$\dim \sigma < n,$ the restriction $\nu_\sigma$
of $\nu_{n-1}$
to the  subcomplex $L_\sigma \cap T \subset L_{n-1} \cap T$
has total space 
$E(\nu_\sigma) = Q_\sigma.$
Let $\sigma \in K$ be an $n$-cell.
The pair $(M_\sigma, \partial M_\sigma)$ is triangulated by
$(L_\sigma \cap T, L_{\partial \sigma} \cap T)$.
Using the description (\ref{equ.boundaryqsigma}) of $\partial Q_\sigma$,
we have
\[ E(\nu_{n-1}|_{L_{\partial \sigma} \cap T}) 
 = \bigcup_{\tau \in \partial \sigma}
    E(\nu_{n-1}|_{L_\tau \cap T})
 = \bigcup_\tau Q_\tau
   \subset \partial Q_\sigma. \]
Since $Q_\sigma$ is a regular neighborhood of
the compact manifold $M_\sigma$, there exists, again 
by \cite[Thm. 4.3]{rosablockbundles1}, a disc block bundle
$\nu_\sigma$ over $L_\sigma \cap T$ with total space
$E(\nu_\sigma) = Q_\sigma$
such that
\[ \nu_\sigma|_{L_{\partial \sigma} \cap T} 
  = \nu_{n-1}|_{L_{\partial \sigma} \cap T}.  \]
Then the collection of blocks
$\nu_\sigma (\beta),$ $\beta \in L_\sigma \cap T,$ $\dim \sigma \leq n,$
of the bundles $\nu_\sigma$ endow 
$Q_n = \bigcup_{\dim \sigma \leq n} Q_\sigma$ with the structure of
a disc block bundle $\nu_n$ over $L_n \cap T$.
By construction,
\[ E(\nu_n|_{L_\sigma \cap T}) = E(\nu_\sigma) = Q_\sigma \]
for all $\sigma \in K,$ $\dim \sigma \leq n$.
This concludes the inductive step.
For $n=b$, 
$\nu_\theta := \nu_b$ is a PL closed disc block bundle
over $L_b \cap T =T$ with total space
$E(\nu) = N.$

If $P$ is an oriented codimension $0$ submanifold of the
boundary $\partial M$ of an oriented manifold $M$, 
then the incidence number $\epsilon (P,M)$ is defined to
be $+1$ if the orientation on $P$ induced by the orientation
of $M$ agrees with the given orientation of $P$, and
$-1$ otherwise. 
Suppose that $\xi$ is oriented as an $F$-block bundle.
Then $K$ is an oriented cell complex and each block
$\xi (\sigma)$ is oriented (as a manifold) such that
$\epsilon (\xi (\tau), \xi (\sigma)) = \epsilon (\tau, \sigma)$
whenever $\tau$ is a codimension $1$ face of a cell $\sigma \in K$.
Requiring $\theta$ to be orientation preserving, we obtain
orientations of all $M_\sigma$ such that
$\epsilon (M_\tau, M_\sigma) = \epsilon (\tau, \sigma)$.
Give every $\real^s \times \sigma$ the product orientation
determined by the orientation of the cell $\sigma$ and the
standard orientation of $\real^s$. Then the inclusion embeddings
of oriented manifolds $M_\sigma \subset \real^s \times \sigma$
induce unique orientations of the normal bundles
$\nu_\sigma$. The above incidence number relation implies that
these orientations fit together to give an orientation of $\nu_\theta$.
\end{proof}

Using the PL homeomorphism $\theta: X \to \theta (X)$, we may think
of $\nu_\theta$ as a bundle over $X$.

\begin{prop} \label{prop.nuthetaindepoftheta}
For $s$ sufficiently large (compared to the dimension of $X$),
the equivalence class of the disc block bundle $\nu_\theta$ 
as constructed in Proposition \ref{thm.existencenormaldiscblockbundle}
is independent of the
choice of blockwise embedding $\theta:X\hookrightarrow \real^s \times B$
and thus only depends on the $F$-block bundle $\xi$.
\end{prop}
\begin{proof}
Let $\theta, \theta': X \hookrightarrow \real^s \times B$ be
$\xi$-block preserving PL embeddings, giving rise to
vertical normal disc block bundles $\nu_\theta$ and $\nu_{\theta'}$.
The idea is to construct a $(\xi \times I)$-block preserving concordance
$\overline{\theta}: X\times I \hookrightarrow \real^s \times B \times I$ 
between $\theta$ and $\theta'$ and
then apply Proposition \ref{thm.existencenormaldiscblockbundle}
to endow inductively a suitable regular neighborhood of the image of the concordance
with the structure of a disc block bundle, extending the
disc block bundles $\nu_\theta$ and $\nu_{\theta'}$.
This implies that $\nu_\theta$ and $\nu_{\theta'}$ are equivalent. 

We begin by observing that the equivalence class of the 
block bundle $\nu_\theta$ does not change under passage to a
simplicial subdivision $L_0$ of
the complex $L$ used in the proof of 
Proposition \ref{thm.existencenormaldiscblockbundle}.
This is a consequence of Cohen's
uniqueness theorem for regular neighborhoods in general polyhedra,
\cite[Thm. (3.1), p. 196]{cohenregnbhdsinpolyhedra} and
Rourke-Sanderson's uniqueness theorem for disc block bundle
structures, \cite[Thm. 4.4, p. 16]{rosablockbundles1}.

The cylinder $\overline{B} = B\times I$ is the polyhedron of 
the product cell complex
$\overline{K} = K \times I$, where $I$ carries the minimal cell structure.
Let $\overline{X} = X\times I$. 
The product block bundle $\overline{\xi} := \xi \times I$
is an $F$-block bundle over the cell complex $\overline{K}$
with total space $E(\xi \times I) = \overline{X}$
and blocks
$(\xi \times I)(\sigma \times \tau) = \xi (\sigma) \times \tau$,
where $\sigma$ is a cell in $K$ and $\tau$ a cell of $I$.
For sufficiently large $s$, by induction over the finitely many cells
$\sigma$ in $\overline{K}$, 
there exists a PL embedding
\[ \overline{\theta}: \overline{X} \longrightarrow 
   \real^s \times \overline{B} = (\real^s \times B) \times I \]
such that
$\overline{\theta}_0 = \theta \times 0,$
$\overline{\theta}_1 = \theta' \times 1,$
$\overline{\theta} 
  (\overline{\xi} (\overset{\circ}{\sigma})) 
  \subset \real^s \times \overset{\circ}{\sigma}$ 
and
\[ \overline{\theta}|: 
  (\overline{\xi} (\sigma), \overline{\xi} (\partial \sigma)) 
  \longrightarrow (\real^s \times \sigma, 
    \real^s \times \partial \sigma) \]
is a locally flat PL embedding of manifolds
for every cell $\sigma \subset \overline{K}$.
Thus $\overline{\theta}$ is a block preserving concordance
between $\theta$ and $\theta'$ satisfying
\[
\overline{\theta}(X\times I) \cap (\real^s \times B \times 0)
  = \theta (X) \times 0,~  
  \overline{\theta}(X\times I) \cap (\real^s \times B \times 1)
  = \theta' (X) \times 1.  
\]  
There exists a locally finite simplicial complex
$\overline{L}$ with subcomplexes $\overline{T}, 
\overline{L}_\sigma \subset \overline{L}$ 
($\sigma \in \overline{K}$)
such that the following
properties hold:\\  
  
\noindent (i) $|\overline{L}| = \real^s \times \overline{B}$,\\

\noindent (ii) the complexes $L\times 0$ and $L' \times 1$ 
 used in constructing $\nu_\theta$ and $\nu_{\theta'}$
are both subcomplexes of $\overline{L}$ such that
\[ |L| = \real^s \times B\times 0,~  
   |L'| = \real^s \times B\times 1, \]

\noindent (iii) $\overline{\theta} (\overline{X}) = |\overline{T}|$,\\

\noindent (iv) for every cell $\sigma$ in $\overline{K}$,
 \[ \real^s \times \sigma = |\overline{L}_\sigma|, \text{ and} \]

\noindent (v) each simplex of $\overline{L}$ meets 
   $\overline{T}$ in a single face or not at all.\\

(To achieve the fullness property (v), it may be necessary
to subdivide $L\times 0$ and $L' \times 1$, but we have
observed earlier that this does not change the equivalence class
of $\nu_\theta, \nu_{\theta'}$. Thus we may call the subdivisions
$L\times 0$ and $L' \times 1$ again.)
Let $f: L \to [0,1]$ be the unique simplicial map
such that $f^{-1} (0) = |T| = \theta (X)$.
The disc block bundle $\nu_\theta$ over $T$ has total space
\[ E(\nu_\theta) = N = f^{-1} [0, 1/2] \subset \real^s \times B, \]
a regular neighborhood of $\theta (X)$ in $\real^s \times B$.
Let $f': L' \to [0,1]$ be the unique simplicial map
such that $f'^{-1} (0) = |T'| = \theta' (X)$.
The disc block bundle $\nu_{\theta'}$ over $T'$ has total space
\[ E(\nu_{\theta'}) = N' = f'^{-1} [0, 1/2] \subset \real^s \times B, \]
a regular neighborhood of $\theta' (X)$ in $\real^s \times B$.
Let $\overline{f}: \overline{L} \to [0,1]$ be the unique simplicial map
such that $\overline{f}^{-1} (0) = |\overline{T}| 
= \overline{\theta} (X\times I)$.
By Proposition \ref{thm.existencenormaldiscblockbundle} and its proof,
the regular neighborhood
\[ \overline{N} := \overline{f}^{-1} [0, 1/2] 
     \subset \real^s \times B \times I \]
of $\overline{\theta} (X\times I)$
is the total space 
$E(\nu_{\overline{\theta}})=\overline{N}$ of a PL disc block bundle
$\nu_{\overline{\theta}}$ over $\overline{T}$ such that
\[ \nu_{\overline{\theta}}|_{L\times 0} = \nu_\theta,~
   \nu_{\overline{\theta}}|_{L' \times 1} = \nu_{\theta'}. \]
Thus, pulling back $\nu_{\overline{\theta}}$ to $X\times I$ along 
$\overline{\theta}$, we obtain a PL disc block bundle
over $X\times I$ whose restriction
to $X\times 0$ is $\nu_\theta$ and whose restriction to
$X\times 1$ is $\nu_{\theta'}$. This implies that
$\nu_\theta$ and $\nu_{\theta'}$ are equivalent as disc block bundles.
\end{proof}

The oriented normal block bundle $\nu_\theta$ provided by 
Theorem \ref{thm.existencenormaldiscblockbundle} is classified by a map
\[ X \longrightarrow \BBSPL_{s-d}. \]
If $s$ is sufficiently large, then by Proposition \ref{prop.nuthetaindepoftheta},
the homotopy class of this map does not depend on the choice of
blockwise embedding $\theta$. We denote the resulting 
disc block bundle equivalence class by $\nu_\xi$ and refer to it as the
\emph{stable vertical normal block bundle} of $\xi$.
The restriction $s> b+2d+1$ ensures that
the block bundle $\nu_\xi$ is in the stable range,
there exists a unique (up to equivalence) 
oriented PL microbundle $\mu$ over $X$ whose
underlying block bundle is $\nu_\xi$, and this microbundle is also in the
stable range: Since $\dim X = d+b < (s-d)-1,$
the  natural map 
\[ [X,\BSPL_{s-d}] \cong [X,\BBSPL_{s-d}] \]
is a bijection.
We will refer to $\mu$ as the \emph{stable vertical normal microbundle}
of $\xi$. 

\begin{example} \label{exple.trivialxiverticalbundle}
For the trivial $F$-block bundle $\xi$ with total space
$X = F \times B$, we may choose a PL embedding
$\theta_F: F \hookrightarrow \real^s$, $s$ large, and take
$\theta: X \hookrightarrow \real^s \times B$ to be
$\theta = \theta_F \times \id_B: F\times B \hookrightarrow \real^s \times B$,
which is $\xi$-block preserving.
Let $\nu_F$ be the (stable) normal disc block bundle of $\theta_F$
and $\mu_F$ its unique lift to a PL microbundle.
Then the stable vertical normal block bundle $\nu_\xi$ is represented by
$\nu_\theta = \pro^*_1 \nu_F$ and the stable vertical normal microbundle 
is $\mu = \pro^*_1 \mu_F,$ where $\pro_1: F\times B \to F$ is the factor projection.
\end{example}

\begin{example}
If $F$ is a point, then $X=B$ and we may take
$\theta: X=B \hookrightarrow \real^s \times B$ to be
$\theta (x) = (0,x)$.
The stable vertical normal block bundle $\nu_\xi$ and
the stable vertical normal microbundle $\mu$ are both trivial.
\end{example}

\section{The PL Umkehr Map}
\label{sec.umkehrmap}

Given an oriented $F$-block bundle $\xi$ with nonsingular fiber $F$
over a compact polyhedron and
a module spectrum $E$ over the Thom spectrum $\MSPL$, we will
construct a transfer homomorphism 
$\xi^!: E_n (B) \longrightarrow E_{n+d} (X).$ This will be done
in Section \ref{sec.bundletransferstable} by composing suspension,
the PL Umkehr map $T(\xi)$ and the Thom homomorphism $\Phi$.
The Umkehr map will be constructed in the present section, and
the Thom homomorphism in the next.

As in the previous section, let $F$ be a closed oriented PL manifold 
of dimension $d$ and let $K$ be a finite ball complex with associated
polyhedron $B=|K|$.
Let $\xi$ be an oriented PL $F$-block bundle over $K$
with total space $X=E(\xi)$. Fix a block preserving PL embedding
$\theta: X \hookrightarrow \real^s \times B$ for sufficiently large $s$
and let us briefly write $\nu$ for the vertical normal disc block bundle 
$\nu_\theta$ given
by Proposition \ref{thm.existencenormaldiscblockbundle}.
As discussed in the previous section, there is a unique PL microbundle $\mu$ 
whose underlying block bundle is $\nu$.
The total space $E(\nu)=N$
is a $\xi$-block preserving
regular neighborhood of $\theta (X)$ in $B\times \real^s$.
Let $\dot{\nu}$ denote the sphere block bundle of $\nu$
and write $\partial N$ for the total space of $\dot{\nu}$. 
Let 
\[ \Th (\nu) := N \cup_{\partial N} \cone (\partial N)  \]
be the Thom space of $\nu$.
The cone point in $\Th (\nu)$ will
be denoted by $\infty$. Thom spaces of PL microbundles have been
constructed by Williamson in \cite{williamson}.
By his construction, we may take $\Th (\mu) = \Th (\nu)$, since
the underlying block bundle of $\mu$ is $\nu$ and the homotopy type
of the Thom space depends only on the underlying block bundle
(in fact only on the underlying spherical fibration).

We shall construct a PL map
\[ T(\xi): S^s B^+ = \Th (\real^s \times B) \longrightarrow \Th (\nu), \]
called the \emph{Umkehr map} following the terminology
of \cite{beckergottlieb}.
Points in $N \subset S^s B^+$ are to be mapped by the identity
to points in $N\subset \Th (\nu)$. 
By M. Cohen's \cite[Theorem 5.3]{cohenregnbhdsinpolyhedra},
$\partial N$ is collared in the closure of $(\real^s \times B) - N$.
Thus there exists a polyhedral neighborhood $V$ of $\partial N$ in
the closure of $(\real^s \times B) - N$ and a PL homeomorphism
$h: (\partial N) \times [0,1] \cong V$ such that
$h(x,0)=x,$ $x\in \partial N$.
Now let $T(\xi)$
map those points of $V$ that lie in
$h((\partial N)\times \{ 1 \})$ to $\infty$.
Map all points in
$S^s B^+ - (N \cup V)$ to $\infty$.
Finally, map the points in $V$, using the collar coordinate in $[0,1]$, 
correspondingly along cone lines in 
$\cone (\partial N) \subset \Th (\nu)$.
This concludes the description of the Umkehr map
$T(\xi): S^s B^+ \to \Th (\nu)$.
Since it sends $\infty$ to $\infty$, this is a pointed map.

\begin{example} \label{exple.trivialxiumkehrmap}
We continue Example \ref{exple.trivialxiverticalbundle} 
on the trivial $F$-block bundle $\xi$.
Let 
$T(F): S^s \to \Th (\nu_F)=\Th (\mu_F)$ be the standard Pontrjagin-Thom collapse
over a point, associated to the embedding $\theta_F:F \hookrightarrow \real^s$.
The Umkehr map for $\xi$ is given by
\[ T(\xi): S^s \Smash B^+ 
  \stackrel{T(F) \Smash \id_{B^+}}{\longrightarrow}
   \Th (\nu_F) \Smash B^+= \Th (\nu_\theta). \]
\end{example}

If $E$ is any spectrum, then on reduced $E$-homology 
the Umkehr map induces a homomorphism
\[ T(\xi)_*: \widetilde{E}_{n+s} (S^s B^+) \longrightarrow 
              \widetilde{E}_{n+s} (\Th (\nu)). \]
The suspension isomorphism provides an identification
\[ \sigma: E_n (B) = \widetilde{E}_n (B^+) \cong \widetilde{E}_{n+s} (S^s B^+). \]
The composition yields a map
\[ T(\xi)_* \circ \sigma: E_n (B) 
   \longrightarrow \widetilde{E}_{n+s} (\Th (\nu))
              = \widetilde{E}_{n+s} (\Th (\mu)). \]

\begin{example} \label{exple.trivialxiumkehronehomol}
We continue Example \ref{exple.trivialxiumkehrmap} on the trivial $F$-block bundle $\xi$.
Let $E$ be a commutative ring spectrum and let
$[S^s]_E \in \widetilde{E}_s (S^s)$ denote the image of the unit
$1 \in \pi_0 (E)$ under 
$\sigma: \widetilde{E}_0 (S^0) \stackrel{\cong}{\longrightarrow}
 \widetilde{E}_s (S^s)$.
Then the above map $T(\xi)_* \circ \sigma$ has the description
\begin{align*}
T(\xi)_* \sigma (a)
&= (T(F) \Smash \id_{B^+})_* \sigma (a) 
  = (T(F) \Smash \id_{B^+})_* ([S^s]_E \Smash a) \\
&= (T(F)_* [S^s]_E) \Smash a,
\end{align*}
where $a\in E_n (B)$.
Setting $[\Th \mu_F]_E = T(F)_* [S^s]_E$,
we thus arrive at
\[ T(\xi)_* \sigma (a) = [\Th \mu_F]_E \Smash a. \]
\end{example}

\section{The Thom Homomorphism, Mock Bundles and Witt Spaces}
\label{sec.thomhomomorphism}

We recall the Thom homomorphism $\Phi$ associated to an
oriented PL microbundle $\mu$. This homomorphism will later be
used in the definition of the $F$-block bundle transfer $\xi^!$
with $\mu$ the stable vertical normal PL microbundle of $\xi$. 
The Thom map is given by cap product with the Thom class of 
$\mu$. Therefore, we will recall the homotopy-theoretic description
$u_\SPL (\mu)$ of this class, as well as its geometric description
$u_{BRS} (\mu_\PLB)$ in terms of mock bundles as given by
Buoncristiano, Rourke and Sanderson in \cite{buonrs},
where $\mu_\PLB$ denotes the underlying PL closed disc block bundle
of $\mu$. In particular, we take the opportunity to provide
a brief review of mock bundle theory. Mock bundles over Witt
spaces will play an important role later on. One key fact in
the subsequent development is that the total space of a mock
bundle over a Witt space is again a Witt space.

Let $\MSPL$ be the Thom spectrum associated to PL microbundles
(or PL $(\real^m,0)$-bundles, Kuiper-Lashof \cite{kuiperlashof}).
This is a ring spectrum whose homotopy groups can be identified with the bordism
groups of oriented PL manifolds via the Pontrjagin-Thom isomorphism.
Let $\gamma^\SPL_m$ denote the universal oriented rank $m$
$\PL$-bundle over the classifying space $\BSPL_m$.
An oriented PL microbundle $\mu$ of rank $m$ over a
compact polyhedron $X$ is classified by a map 
$X\to \BSPL_m$,
which is covered by a bundle map $\mu \to \gamma^\SPL_m$.
The induced map on Thom spaces yields
a homotopy class 
\[ u_\SPL (\mu) \in [\Sigma^\infty \Th (\mu), \Sigma^m \MSPL]
  = \RMSPL^m (\Th (\mu)). \]
This class $u_\SPL (\mu)$ is the
\emph{Thom class} of $\mu$ in oriented PL cobordism.
It is in fact an $\MSPL$-orientation of $\mu$
in Dold's sense. Indeed,  
every $H\intg$-orientable $PL$-bundle is $\MSPL$-orientable
(Hsiang-Wall \cite[Lemma 5, p. 357]{hsiangwall}, 
Switzer \cite[p. 308]{switzer}).

Buoncristiano, Rourke and Sanderson
give a geometric description of $\MSPL$-cobordism in \cite{buonrs}
and use it to obtain in particular a geometric description of the Thom
class $u_\SPL (\mu).$ 
The geometric cocycles are given by oriented mock bundles, whose
definition we recall here.
\begin{defn}
Let $K$ be a finite ball complex and $q$ an integer (possibly negative).
A \emph{$q$-mock bundle} $\eta^q/K$ with base $K$ and
total space $E(\eta)$ consists of a PL map
$p:E(\eta)\to |K|$ such that, for each $\sigma \in K$,
$p^{-1} (\sigma)$ is a compact PL manifold of dimension
$q + \dim (\sigma),$ with boundary $p^{-1} (\partial \sigma)$.
The preimage $\eta (\sigma) := p^{-1} (\sigma)$ is called the
\emph{block} of $\eta$ over $\sigma$.
\end{defn}
The empty set is regarded as a manifold of
any dimension; thus $\eta (\sigma)$ may be empty for some
cells $\sigma \in K$.
Note that if $\sigma^0$ is a $0$-dimensional cell of $K$,
then $\partial \sigma^0 = \varnothing$ and thus
$p^{-1} (\partial \sigma)=\varnothing$.
Hence the blocks over $0$-dimensional cells are \emph{closed} manifolds.
Mock bundles over the same complex are \emph{isomorphic} if
there exists a block-preserving PL homeomorphism of total spaces.
(The homeomorphism is \emph{not} required to preserve the projections.)
For our purposes, we need \emph{oriented}
mock bundles, which are defined using incidence numbers of cells and blocks:
Suppose that $(M^n,\partial M)$ is an oriented PL manifold
and $(N^{n-1}, \partial N)$ is an oriented PL manifold
with $N \subset \partial M$. Then an incidence number
$\epsilon (N,M)=\pm 1$ is defined by comparing the orientation
of $N$ with that induced on $N$ from $M$;
$\epsilon (N,M)=+1$ if these orientations agree and $-1$ if they
disagree. An \emph{oriented cell complex} $K$ is a cell complex
in which each cell is oriented. We then have the incidence
number $\epsilon (\tau, \sigma)$ defined for codimension $1$ faces
$\tau < \sigma \in K$.
\begin{defn}
An \emph{oriented mock bundle} is a mock bundle $\eta/K$
over an oriented (finite) ball complex $K$ in which every block
is oriented (i.e. is an oriented PL manifold) such that
for each codimension $1$ face $\tau$ of $\sigma \in K$,
$\epsilon (\eta (\tau), \eta (\sigma)) = \epsilon (\tau, \sigma)$.
\end{defn}

Using intersection homology, Witt spaces have been introduced by 
P. Siegel in \cite{siegel}
as a geometric cycle reservoir representing $\operatorname{KO}$-homology 
at odd primes. Sources on intersection homology include
\cite{gmih1}, \cite{gmih2}, \cite{kirwanwoolf}, 
\cite{friedmanihbook}, \cite{banagltiss}.
\begin{defn}
A \emph{Witt space} is an oriented PL pseudomanifold such that
the links $L^{2k}$ of odd codimensional PL intrinsic strata have
vanishing middle-perversity degree $k$ rational intersection homology,
$IH^{\bar{m}}_k (L;\rat)=0$.
\end{defn}
For example, pure-dimensional complex algebraic varieties are Witt spaces,
since they are oriented pseudomanifolds and 
possess a Whitney stratification whose strata all have
even codimension. The vanishing condition on the intersection
homology of links $L^{2k}$ is equivalent to requiring the canonical morphism
from lower middle to upper middle perversity intersection chain sheaves
to be an isomorphism in the derived category of sheaf complexes.
Consequently, these middle perversity intersection chain sheaves are
Verdier self-dual, and this induces global Poincar\'e duality for
the middle perversity intersection homology groups of a Witt space.
In particular, Witt spaces $X$ have a well-defined bordism invariant 
signature and $L$-classes $L_* (X)\in H_* (X;\rat)$ which agree with
the Poincar\'e duals of Hirzebruch's tangential $L$-classes when $X$ is smooth.
The notion of \emph{Witt spaces with boundary} 
can be introduced as pairs
$(X,\partial X),$ where $X$ is a PL space and $\partial X$ a 
stratum preservingly collared PL subspace of $X$ such that
$X-\partial X$ and $\partial X$ are both compatibly oriented Witt spaces.
The following result is \cite[Lemma 3.11]{banaglnyjm}, which is itself
an analog of \cite[Lemma 1.2, p. 21]{buonrs}).

\begin{lemma} \label{lem.mockbundleoverwittiswitt}
Let $(K,L)$ be a finite ball complex pair such that the polyhedron
$|K|$ is an $n$-dimensional compact Witt space with (possibly empty) boundary
$\partial |K|=|L|$.
Orient $K$ in such a way that the sum of oriented $n$-balls
is a cycle rel boundary. (This is possible since $|K|$,
being a Witt space, is an oriented
pseudomanifold with boundary.)
Let $\eta/K$ be an oriented $q$-mock bundle over $K$ with projection $p$.
Then the total space $E(\eta)$ is an
$(n+q)$-dimensional compact Witt space with boundary 
$p^{-1} (\partial |K|).$
\end{lemma}

Let $(K,L)$ be any finite ball complex pair.
Oriented mock bundles $\eta_0$ and $\eta_1$ over $K$,
both empty over $L$,
are \emph{cobordant}, if there is an oriented mock bundle
$\eta$ over $K\times I$, empty over $L\times I$,
such that $\eta|_{K\times 0} \cong \eta_0$,
$\eta|_{K\times 1} \cong \eta_1$.
This is an equivalence relation and we set
\[ \Omega^q_\SPL (K,L) :=
 \{ [\eta^q/K] :~ \eta|_L = \varnothing \}, \]
where $[\eta^q/K]$ denotes the cobordism class of the 
oriented $q$-mock bundle $\eta^q/K$ over $K$.
Then by the duality theorem \cite[Thm. II.3.3]{buonrs}
of Buoncristiano, Rourke and Sanderson,
Spanier-Whitehead duality, together with the
Pontrjagin-Thom isomorphism, provides an isomorphism
\begin{equation} \label{equ.identmsplcohbrstq} 
\beta: \Omega^{-q}_\SPL (K,L) \cong\MSPL^q (|K|,|L|) 
\end{equation}
for compact $|K|, |L|$, which is natural with respect to
inclusions $(K', L')\subset (K,L)$;
see also \cite[Remark 3, top of p. 32]{buonrs}.
This is the geometric description of oriented
PL cobordism that we will use throughout the paper.
The functor $\Omega^*_\SPL (-)$ gives rise to a functor
on the category of compact polyhedral pairs and homotopy
classes of continuous maps, which will be denoted by the same symbol
(\cite[Thm. II.1.1]{buonrs}).

Let $\alpha:|K|\to \BBSPL_m$ be an oriented PL closed disc block bundle
of rank $m$ over a finite complex $K$. 
Let $N$ denote the total space of $\alpha$ and $\partial N$ the total space
of the sphere block bundle of $\alpha$.
Then $\alpha$ has a Thom class
(cf. \cite[p. 26]{buonrs})
\begin{equation} \label{equ.defubrs} 
u_{BRS} (\alpha) \in \Omega^{-m}_\SPL (N,\partial N),  
\end{equation}   
which we shall call the
\emph{BRS-Thom class} of $\alpha$, given as follows:
Let $i: K \to N$ be the zero section of $\alpha$.
Endow $N$ with the ball complex structure given by taking the
blocks $\alpha (\sigma)$ of the bundle $\alpha$ as balls,
together with the balls of a suitable ball complex structure on the total
space $\partial N$ of the sphere block bundle of $\alpha$.
Then $i: K\to N$ is the projection of an oriented 
$(-m)$-mock bundle $\eta$, and thus determines an element
\[ u_{BRS} (\alpha) = [\eta] \in \Omega^{-m}_\SPL (N,\partial N). \]
The block of $\eta$ over a ball $\alpha (\sigma)$ of $N$ is $\sigma \in K$.
The following lemma is Lemma 3.14 in \cite{banaglnyjm}.
\begin{lemma} \label{lem.usplgoestoubrs}
Let $\alpha:|K|\to \BSPL_m$ be an oriented PL $(\real^m,0)$-bundle, $|K|$ compact.
Under the isomorphism $\beta$ in (\ref{equ.identmsplcohbrstq}),
the BRS-Thom class $u_{BRS} (\alpha_\PLB)$ 
of the underlying oriented PL block bundle $\alpha_\PLB$ of $\alpha$
gets mapped to the Thom class $u_\SPL (\alpha)$ .
\end{lemma}

Let $E$ be an $\MSPL$-module spectrum. Then there is a cap product
\[ \cap: \MSPL^p (X,A) \otimes E_q (X,A) \longrightarrow E_{q-p} (X).  \]
The reduced cobordism group of the Thom space can be expressed as a relative group,
\[ \RMSPL^p (\Th (\mu)) \cong \MSPL^p (N,\partial N), \]
where $N$, as in Section \ref{sec.umkehrmap},
is the total space of the underlying oriented PL closed disc block bundle of $\mu$.
Let
\[ \rho_*: E_* (N) \stackrel{\cong}{\longrightarrow} E_* (X) \]
be the inverse of the isomorphism induced on $E$-homology by the inclusion
$X \hookrightarrow N$ of the zero section.
Using the cap product
\[ \cap: \MSPL^m (N,\partial N) \otimes E_q (N,\partial N) 
   \longrightarrow E_{q-m} (N)
   \stackrel{\rho_*}{\cong} E_{q-m} (X),  \]
we obtain the \emph{Thom homomorphism}
\[   
\Phi := \rho_* (u_\SPL (\mu) \cap -):
\widetilde{E}_q (\Th (\mu)) \cong E_q (N,\partial N) \longrightarrow
 E_{q-m} (X).
\]
Under suitable conditions this map is an isomorphism, for example
if $X$ is connected, $E$ is a ring spectrum and
$u_\SPL (\mu)$ determines an $E$-orientation of $\mu$
(\cite[p. 309, Thm. 14.6]{switzer}; recall that our $X$ is a finite complex).

\section{Block Bundle Transfer}
\label{sec.bundletransferstable}

Let $E$ be a module spectrum over the Thom spectrum $\MSPL$
of oriented $\PL$-bundle theory.
As in Section \ref{sec.stableverticalblockbundles},
$F$ denotes a closed oriented PL manifold of dimension $d$ and
$K$ a finite ball complex with associated
polyhedron $B=|K|$ of dimension $b$.
Let $\xi$ be an oriented PL $F$-block bundle over $K$
with total space $X=E(\xi)$.
Following Boardman \cite{boardman} and 
Becker-Gottlieb \cite{beckergottlieb}, we shall construct a
\emph{transfer homomorphism}
\begin{equation} \label{equ.etransferofp}
\xi^!: E_n (B) \longrightarrow E_{n+d} (X). 
\end{equation}
Let $\mu$ denote the stable oriented vertical normal PL microbundle of $\xi$
whose underlying disc block bundle 
is $\nu_\theta$, the oriented vertical normal disc block bundle of 
the $F$-block bundle $\xi$,
associated to a block preserving embedding $\theta$ for $\xi$.
The rank of $\mu$ and $\nu_\theta$ is $m=s-d$, $d=\dim F,$ $s$ large.
The block bundle transfer is defined to be the composition
\[ E_n (B) \stackrel{T(\xi)_* \sigma}{\longrightarrow}
   \widetilde{E}_{n+s} (\Th (\mu))
   \stackrel{\Phi}{\longrightarrow}
   E_{n+d} (X), \]
where $\sigma$ is the suspension isomorphism, $T(\xi)$ is the
Umkehr map of Section \ref{sec.umkehrmap},
and $\Phi$ is the Thom homomorphism of $\mu$ as described in 
Section \ref{sec.thomhomomorphism}.   
In the present paper, we are mainly interested in the case where $E$ is
ordinary homology, Ranicki's symmetric $\syml$-spectrum, or Witt bordism.   
Let us discuss each of these cases in turn.

\subsection{Block Transfer on Ordinary Homology}

Let $H\intg$ denote the Eilenberg-MacLane spectrum of the ring $\intg$.
The stable universal Thom class in $H^0 (\MSPL)$ yields a map
$\alpha: \MSPL \to H\intg$, and this map is a ring map.
Thus $\alpha$ makes the ring spectrum $H\intg$ 
into an $\MSPL$-module by taking the action map to be
\[ \MSPL \wedge H\intg
   \stackrel{\alpha \wedge \id}{\longrightarrow}
   H\intg \wedge H\intg
   \stackrel{\mu_H}{\longrightarrow}
   H\intg, \]
where $\mu_H$ is the product on $H\intg$.   
The induced map
\[ \alpha_*: \Omega^\SPL_n (Z) \cong \MSPL_n (Z)
  \longrightarrow H_n (Z;\intg) \]
is the Steenrod-Thom homomorphism that sends the bordism class of 
a singular PL manifold 
$[f:M^n \to Z] \in \Omega^\SPL_n (Z)$
to $f_* [M] \in H_n (Z;\intg)$.
We recall the following standard fact.
\begin{prop}
(Rudyak \cite[Prop. V.1.6]{rudyak}.)
Let $\tau: D\to E$ be a ring morphism of ring spectra.
Let $\gamma$ be an $(S^n,*)$-fibration equipped with a
$D$-orientation $u_D \in \widetilde{D}^n (\Th \gamma)$.
Then the image $\tau (u_D) \in \widetilde{E}^n (\Th \gamma)$
is an $E$-orientation of $\gamma$.
\end{prop}
We apply this Proposition to the ring morphism
$\alpha: \MSPL \to H\intg$ and to our microbundle $\mu$, which we
had already equipped with the $\MSPL$-orientation
$u_\SPL (\mu)$. By the Proposition,
the homomorphism
\[ \alpha: \RMSPL^{s-d} (\Th (\mu)) \longrightarrow 
     \widetilde{H}^{s-d} (\Th (\mu);\intg) \]
sends $u_\SPL (\mu)$ to
an $H\intg$-orientation
\begin{equation} \label{lem.uspltouz}
u_\intg (\mu) := \alpha (u_\SPL (\mu)) 
  \in \widetilde{H}^{s-d} (\Th (\mu);\intg).
\end{equation}
(This is the Thom class of $\mu$ in ordinary cohomology.)
Another standard fact from stable homotopy theory is:
\begin{lemma} \label{lem.tderingmorcapprods}
Let $D,E$ be ring spectra and $\tau:D\to E$ a ring morphism.
We consider $E$ as a $D$-module via the action map
\[  D \wedge E
   \stackrel{\tau \wedge \id}{\longrightarrow}
   E \wedge E
   \stackrel{\mu_E}{\longrightarrow} E. \] 
This module structure yields a cap product
$\cap_{D,E}: D^* (X) \otimes E_* (X) \to E_* (X).$   
The ring structure on $E$ yields a cap product
$\cap_E: E^* (X) \otimes E_* (X) \to E_* (X).$ 
Then the diagram
\[ \xymatrix@C=40pt{
D^* (X) \otimes E_* (X) \ar[r]^{\cap_{D,E}} \ar[d]_{\tau \otimes \id} 
  & E_* (X) \ar@{=}[d] \\
E^* (X) \otimes E_* (X) \ar[r]_{\cap_E} & E_* (X)
} \]
commutes.
\end{lemma}
By this lemma and (\ref{lem.uspltouz}),
the transfer on ordinary homology is given by
\begin{align*}
\xi^! (-)
&= \rho_* (u_\SPL (\mu) \cap_{\MSPL, H\intg} T(\xi)_* \sigma (-)) \\
&= \rho_* (\alpha (u_\SPL (\mu)) \cap_{H\intg} T(\xi)_* \sigma (-)) \\
&= \rho_* (u_\intg (\mu) \cap_{H\intg} T(\xi)_* \sigma (-)).
\end{align*}
We summarize: The block bundle transfer (\ref{equ.etransferofp}) on ordinary
homology $E=H\intg$ is given by
\[ \xi^! = \rho_* (u_\intg (\mu) \cap T(\xi)_* \sigma (-)): H_n (B;\intg) 
   \longrightarrow H_{n+d} (X;\intg). \]

\subsection{Block Transfer on Witt Bordism}

Let $\Omega^\Witt_* (-)$ denote Witt bordism theory as defined
by Siegel in \cite{siegel}. Elements of $\Omega^\Witt_n (Z)$ are
Witt bordism classes of continuous maps $f: W^n \to Z$ defined on
an $n$-dimensional closed Witt space $W$.
Let $\MWITT$ be the Quinn spectrum associated to the ad-theory of Witt spaces,
representing Witt bordism via a natural equivalence
\begin{equation} \label{equ.mwittisomegawitt} 
 \MWITT_* (-) \cong
   \Omega^\Witt_* (-),
\end{equation}   
see Banagl-Laures-McClure \cite{blm}.
A weakly equivalent spectrum had been considered 
first by Curran in \cite{curran}. He verified that this spectrum is an
$\MSO$-module (\cite[Thm. 3.6, p. 117]{curran}).
The product of
two Witt spaces is again a Witt space. This implies essentially that
$\MWITT$ is a ring spectrum; for more details see \cite{blm}.
(There, we focused on the spectrum
$\operatorname{MIP}$ representing
bordism of integral intersection homology Poincar\'e spaces studied by
Goresky and Siegel in \cite{gorsie} and by Pardon in \cite{pardon}, 
but everything works in an analogous, indeed simpler, manner for $\rat$-Witt spaces.)
Every oriented PL manifold is a Witt space. Hence there is a map
\[ \phi_W: \MSPL \longrightarrow \MWITT, \]
which, using the methods of ad-theories and Quinn spectra
employed in \cite{blm}, can be constructed to be multiplicative.
Using this ring map, the spectrum $\MWITT$
becomes an $\MSPL$-module with action map
\[ \MSPL \wedge \MWITT \longrightarrow \MWITT \]
given by the composition
\[ \MSPL \wedge \MWITT \stackrel{\phi_W \wedge \id}{\longrightarrow}
   \MWITT \wedge \MWITT \longrightarrow \MWITT. \]
(The product of a Witt space
and an oriented PL manifold is again a Witt space.)
In particular, there is a cap product
\begin{equation} \label{equ.capprodmsplmwitt}
\cap: \MSPL^j (Z,Y) \otimes \MWITT_n (Z,Y) 
   \longrightarrow \MWITT_{n-j} (Z) 
\end{equation}   
and a transfer
\[ \xi^!: \MWITT_n (B) \longrightarrow \MWITT_{n+d} (X), \]
where $\xi$ is our $F$-block bundle over $B$, $d=\dim F$.

Let $C$ be any finite ball complex with subcomplex $D\subset C$
and suppose that $Z=|C|,$ $Y=|D|$.
By Buoncristiano-Rourke-Sanderson \cite{buonrs}, 
a geometric description of the above cap product
(\ref{equ.capprodmsplmwitt}) is given as follows:
One uses the canonical identifications to think
of the cap product as a product
\[ \cap: \Omega_\SPL^{-j} (C,D) \otimes \Omega^\Witt_n (|C|,|D|) \longrightarrow
  \Omega^\Witt_{n-j} (|C|). \]
Let us first discuss the absolute case $D=\varnothing$, and then
return to the relative one.
If $C$ is simplicial, $f: W\to C$ is a simplicial map from an
$n$-dimensional triangulated closed Witt space $W$ to
$C$, and $\eta^q$ is a $q$-mock bundle over $C$ (with $q=-j$), 
then one has (cf. \cite[p. 29]{buonrs})
\[ [\eta^q/C] \cap [f:W\to |C|] =
  [g: E(f^* \eta)\to |C|] \in \Omega^\Witt_{n-j} (|C|),  \]
where $g$ is the diagonal arrow in the cartesian diagram
\[ \xymatrix{
E(f^* \eta) \ar[r] \ar[d] \ar[rd]^g & E(\eta) \ar[d]^p \\
W \ar[r]^f & C.
} \]
Here, one uses the fact (\cite[II.2, p. 23f]{buonrs}) that mock
bundles over simplicial complexes admit pullbacks under simplicial maps.
By Lemma \ref{lem.mockbundleoverwittiswitt},
$E(f^* \eta)$ is a closed Witt space.
For the relative case, we observe that
if $(W,\partial W)$ is a compact Witt space with boundary,
$f:(W,\partial W)\to (|C|,|D|)$ maps the boundary into $|D|$,
and $\eta|_D = \varnothing$, then
$f^* \eta|_{\partial W} = \varnothing$
and so $\partial E(f^* \eta) = \varnothing,$
i.e. the Witt space $E(f^* \eta)$ is \emph{closed}.
Hence it defines an absolute bordism class.

In Section \ref{sec.pullbacktransfer},
we provide a more direct geometric description of the 
Witt bordism transfer
$\xi^!: \MWITT_n (B) \to \MWITT_{n+d} (X)$
as a pullback transfer 
$\xi^!_\PB: \Omega^\Witt_n (B) \to \Omega^\Witt_{n+d} (X)$.

\subsection{Block Transfer on $\syml$-Homology}

We write $\syml = \syml (\intg) = \syml \langle 0 \rangle (\intg)$ for Ranicki's
connected symmetric algebraic $L$-spectrum
with homotopy groups $\pi_n (\syml)=L^n (\intg),$ the symmetric
$L$-groups of the ring of integers; see e.g. \cite{ranickialtm}. 
Technically, we shall use the construction of $\syml$ as the Quinn spectrum of
a suitable ad-theory, see 
Banagl-Laures-McClure \cite{blm}. 
That construction is weakly equivalent to Ranicki's.
Localization $\intg \to \rat$ induces a map
$\epsilon_\rat:\syml (\intg) \to \syml (\rat),$ 
and $\pi_n (\syml (\rat))=L^n (\rat)$ with
\[ L^n (\rat) \cong \begin{cases}
 \intg \oplus (\intg/_2)^\infty \oplus (\intg/_4)^\infty,& n\equiv 0 (4) \\
   0,& n\not\equiv 0 (4).
\end{cases} \]
The spectra $\syml (\intg)$ and $\syml (\rat)$ are ring spectra.
Let $\MSTOP$ be the Thom spectrum associated to oriented topological
$(\real^n, 0)$-bundles. There is a canonical forget map
\[ \phi_F: \MSPL \longrightarrow \MSTOP. \]
Ranicki constructed a map
\[ \sigma^*: \MSTOP \longrightarrow \syml, \]
see \cite[p. 290]{ranickitotsurgob}, and 
in \cite{blm}, we constructed a map
\[ \tau: \MWITT \longrightarrow \syml (\rat). \]
(Actually, we even constructed an integral map
$\operatorname{MIP} \to \syml,$ where $\operatorname{MIP}$ represents
bordism of integral intersection homology Poincar\'e spaces,
but everything works in the same manner for Witt theory, if one
uses the $\syml$-spectrum with rational coefficients.)
This map is multiplicative, i.e. a ring map, as shown
in \cite[Section 12]{blm}, and the diagram
\begin{equation} \label{equ.esigmaphifistauphiw}
\xymatrix{
& \MSTOP \ar[r]^{\sigma^*} & \syml (\intg) \ar[dd]^{\epsilon_\rat} \\
\MSPL \ar[ru]^{\phi_F} \ar[rd]_{\phi_W} & & \\
& \MWITT \ar[r]_\tau & \syml (\rat)
} \end{equation}
homotopy commutes, since it comes from a
commutative diagram of ad-theories under applying
the symmetric spectrum functor $\mathbf{M}$ of 
Laures and McClure \cite{lauresmcclure}.
Using the ring map $\tau \phi_W: \MSPL \to \syml (\rat)$, 
the spectrum $\syml (\rat)$
becomes an $\MSPL$-module with action map
\[ \MSPL \wedge \syml (\rat) \longrightarrow \syml (\rat) \]
given by the composition
\[ \MSPL \wedge \syml (\rat) 
   \stackrel{(\tau \phi_W) \wedge \id}{\longrightarrow}
   \syml (\rat) \wedge \syml (\rat) \longrightarrow \syml (\rat). \]
The associated transfer is
\[ \xi^!: \syml (\rat)_n (B) \longrightarrow \syml (\rat)_{n+d} (X), \]
with $\xi$ our $F$-block bundle over $B$, $d=\dim F$. \\

We shall show that the block bundle transfer $\xi^!$ commutes
with the passage, under $\tau_*$, from Witt bordism theory to 
$\syml (\rat)$-homology.
The homotopy commutative diagram
\[ \xymatrix{
\MSPL \wedge \MWITT \ar[r]^{\id \wedge \tau} \ar[d]_{\phi_W \wedge \id} & 
   \MSPL \wedge \syml (\rat) \ar[d]^{(\tau \phi_W)\wedge \id} \\
\MWITT \wedge \MWITT \ar[r]^{\tau \wedge \tau} \ar[d] 
   & \syml (\rat) \wedge \syml (\rat) \ar[d] \\
\MWITT \ar[r]^\tau & \syml (\rat)
} \]
shows that $\tau: \MWITT \to \syml (\rat)$ is an
$\MSPL$-module morphism.
In the proof of Lemma \ref{lem.phitau} below,
we shall use the following standard fact:
\begin{lemma} \label{lem.caponemodmor}
If $E$ is a ring spectrum, $F,F'$ module spectra over $E$ and
$\phi: F\to F'$ an $E$-module morphism, then the diagram
\[ \xymatrix{
E^m (X,A) \otimes F_n (X,A) \ar[r]^>>>>>\cap  \ar[d]_{\id \otimes \phi_*}
  & F_{n-m} (X) \ar[d]^{\phi_*} \\
E^m (X,A) \otimes F'_n (X,A) \ar[r]^>>>>>\cap & F'_{n-m} (X)
} \]
commutes: if $u\in E^m (X,A),$ and $a\in F_n (X,A)$, then
\[ \phi_* (u\cap a) = u\cap \phi_* (a). \]
\end{lemma} 

\begin{lemma} \label{lem.phitau}
The Thom homomorphisms $\Phi$ of an oriented PL microbundle $\mu$
of rank $m$ over a compact polyhedron $X$
commute with the passage from
Witt bordism to $\syml (\rat)$-homology, that is, the diagram
\[ \xymatrix{
\widetilde{\MWITT}_n (\Th (\mu)) \ar[r]^\Phi \ar[d]_{\tau_*}
  & \MWITT_{n-m} (X) \ar[d]_{\tau_*} \\
\redsyml (\rat)_n (\Th (\mu)) \ar[r]^\Phi 
  & \syml (\rat)_{n-m} (X)
} \]
commutes.
\end{lemma}
\begin{proof}
As $\tau_*$ is a natural transformation of homology theories,
it commutes with the isomorphism $\rho_*$.
Since $\tau$ is an $\MSPL$-module morphism, Lemma \ref{lem.caponemodmor}
applies to give
\begin{align*}
\tau_* \Phi
&= \tau_* \rho_* (u \cap -) = \rho_* \tau_* (u \cap -) \\
&= \rho_* (u \cap \tau_* (-)) = \Phi \tau_*,
\end{align*}
where $u=u_\SPL (\mu)$.
\end{proof}

\begin{prop} \label{prop.transftaucomm}
Let $F$ be a closed oriented $d$-dimensional PL manifold and let 
$\xi$ be an oriented $F$-block bundle with total space $X$ over the
compact polyhedral base $B$.
Then the diagram
\[ \xymatrix{
\MWITT_n (B) \ar[r]^{\xi^!} \ar[d]_{\tau_*}
  & \MWITT_{n+d} (X) \ar[d]_{\tau_*} \\
\syml (\rat)_n (B) \ar[r]^{\xi^!} 
  & \syml (\rat)_{n+d} (X)
} \]
commutes.
\end{prop}
\begin{proof}
Let $\mu$ be the stable vertical normal microbundle of $\xi$.
The right hand square of the diagram
\[ \xymatrix{
\MWITT_n (B) \ar[r]^<<<<{\cong}_<<<<\sigma \ar[d]_{\tau_*}
  & \widetilde{\MWITT}_{n+s} (S^s B^+) \ar[d]_{\tau_*} \ar[r]^{T(\xi)_*}
  & \widetilde{\MWITT}_{n+s} (\Th (\mu)) \ar[d]^{\tau_*} \\
\syml (\rat)_n (B) \ar[r]^<<<<<{\cong}_<<<<<\sigma
  & \redsyml (\rat)_{n+s} (S^s B^+) \ar[r]^{T(\xi)_*} 
  & \redsyml (\rat)_{n+s} (\Th (\mu))
} \]
commutes, as $\tau_*$ is a natural transformation of homology theories.
The left hand square, involving the suspension isomorphism $\sigma$,
commutes for the same reason.
The statement now follows from Lemma \ref{lem.phitau}.
\end{proof}

An oriented topological $(\real^m, 0)$-bundle $\alpha$ over a
CW complex $Z$, classified by a map $Z\to \BSTOP_m$,
possesses a Thom class 
\[ u_\STOP (\alpha) \in \RMSTOP^m (\Th (\alpha)) \]
in oriented topological cobordism.
The next auxiliary result on compatibility of Thom classes is standard and 
e.g. recorded as Lemma 3.7 in \cite{banaglnyjm}.
\begin{lemma} \label{lem.usplustop}
Let $\alpha$ be an oriented PL $(\real^m,0)$-bundle.
On cobordism groups, the homomorphism
\[ \phi_F: \RMSPL^m (\Th (\alpha)) \longrightarrow
              \RMSTOP^m (\Th (\alpha_\TOP)) \]
induced by the canonical map 
$\phi_F: \MSPL \to \MSTOP$
sends the Thom class of $\alpha$ to the Thom class of the
underlying oriented topological $(\real^m,0)$-bundle $\alpha_\STOP$,
\[  \phi_F (u_\SPL (\alpha)) = u_\STOP (\alpha_\STOP). \]              
\end{lemma}

Following \cite[pp. 290, 291]{ranickitotsurgob}, 
an oriented topological $(\real^m,0)$-bundle
$\alpha$ has a canonical $\syml$-cohomology orientation
\[ u_\mbl (\alpha) \in {\redsyml}^m (\Th (\alpha)), \]
which we shall also refer to as the 
\emph{$\syml$-cohomology Thom class} of $\alpha$, defined by
\begin{equation} \label{equ.umblissigmaustop}
u_\mbl (\alpha) := \sigma^* (u_\STOP (\alpha)). 
\end{equation}
The morphism of spectra 
$\epsilon_\rat: \syml (\intg)\to \syml (\rat)$ coming from
localization induces a homomorphism
\[ \redsyml^m (\Th (\alpha)) \longrightarrow 
   \redsyml (\rat)^m (\Th (\alpha)). \]
We denote the image of $u_\mbl (\alpha)$ under this map again by
$u_\mbl (\alpha) \in \redsyml (\rat)^m (\Th (\alpha)).$

\begin{lemma} \label{lem.umsotoumbl}
Let $\alpha$ be an oriented PL $(\real^m,0)$-bundle.
The homomorphism
\[ \tau \phi_W: 
  \RMSPL^m (\Th (\alpha)) \longrightarrow
              \syml (\rat)^m (\Th (\alpha)) \]
induced by the ring morphism 
$\tau \phi_W: \MSPL \to \syml (\rat)$
sends the $\MSPL$-cohomology Thom class of $\alpha$ to the 
$\syml$-cohomology Thom class of (the underlying topological bundle of) $\alpha$,
\[ \tau \phi_W (u_\SPL (\alpha)) = u_\mbl (\alpha). \]  
\end{lemma}
\begin{proof}
By Lemma \ref{lem.usplustop}, 
Ranicki's definition
(\ref{equ.umblissigmaustop}), and the homotopy commutativity of diagram
(\ref{equ.esigmaphifistauphiw}),
\begin{align*}
\tau \phi_W (u_\SPL (\alpha))
&= \epsilon_\rat \sigma^* \phi_F (u_\SPL (\alpha)) 
   = \epsilon_\rat \sigma^* (u_\STOP (\alpha_\STOP)) \\
&= u_\mbl (\alpha_\STOP). 
\end{align*}
\end{proof}
Lemma \ref{lem.umsotoumbl}, together with Lemma \ref{lem.tderingmorcapprods},
implies that the $F$-block bundle transfer on 
$\syml (\rat)$-homology is given by
\[ \xi^! = \rho_* (u_\mbl (\mu) \cap T(\xi)_* \sigma (-)): 
   \syml (\rat)_n (B) \longrightarrow \syml (\rat)_{n+d} (X). \]

\begin{example} \label{exple.trivialxitransfer}
We continue Example \ref{exple.trivialxiumkehronehomol} 
and compute the transfer for the trivial $F$-block bundle $\xi$
with total space $X=F \times B$.
Let $E$ be a commutative ring spectrum and 
$\phi: \MSPL \to E$ a morphism of ring spectra, equipping $E$ with the structure
of an $\MSPL$-module.
Recall that we had chosen a PL embedding $\theta_F: F \hookrightarrow \real^s$
with $s$ large enough so that $\theta_F$ has a tubular neighborhood given by
a PL microbundle $\mu_F$ which represents the stable normal PL microbundle of $F$.
The stable vertical normal bundle of $\xi$ is then given by
$\mu=\pro^*_1 \mu_F$.
Its Thom class $u_\SPL (\mu) \in \widetilde{\MSPL}^{s-d} (\Th (\mu))
  = \widetilde{\MSPL}^{s-d} (\Th (\mu_F) \Smash B^+)$
is $u_\SPL (\mu) = u_\SPL (\mu_F) \Smash 1$, 
since the bundle map $\mu \to \gamma^\SPL_{s-d}$  
factors as $\mu \to \mu_F \to \gamma^\SPL_{s-d},$  
where the first map covers the projection $\pro_1: F\times B\to F$
and the second map the classifying map for $\mu_F$.
The element $\phi (u_\SPL (\mu_F))$ is an $E$-orientation of $\mu_F$
(\cite[Prop. V.1.6]{rudyak}) and thus
$[F]_E := \rho_{F*} 
   (\phi (u_\SPL (\mu_F)) \cap [\Th \mu_F]_E) \in E_d (F)$,
with $\rho_{F*}: E_d (N_F) \cong E_d (F)$,   
is an $E$-homology orientation for the PL manifold $F$
(\cite[Prop. V.2.8]{rudyak}, \cite[p. 333, Lemma 14.40]{switzer}).  
The transfer $\xi^!: E_n (B) \to E_{n+d} (F\times B)$ is then given on 
$a\in E_n (B)$ by
\begin{equation} \label{equ.transfertrivialxiisproduct}
\xi^! (a) = [F]_E \times a,
\end{equation}
as follows from the calculation
\begin{align*}
\xi^! (a)
&= \Phi T(\xi)_* \sigma (a) \\
&= \Phi ([\Th \mu_F]_E \Smash a) \\
&= \rho_* (\phi(u_\SPL (\mu)) \cap ([\Th \mu_F]_E \Smash a)) \\
&= \rho_* (\phi(u_\SPL (\mu_F) \Smash 1) \cap ([\Th \mu_F]_E \Smash a)) \\
&= \rho_* ((\phi (u_\SPL (\mu_F)) \Smash 1) \cap ([\Th \mu_F]_E \Smash a)) \\
&= \rho_* ((\phi (u_\SPL (\mu_F)) \cap [\Th \mu_F]_E) 
                     \times (1 \cap a)) \\
&= \rho_{F*} (\phi (u_\SPL (\mu_F)) \cap [\Th \mu_F]_E)
         \times a \\
&= [F]_E \times a.
\end{align*}  
\end{example}

\section{Geometric Pullback Transfer on Bordism}
\label{sec.pullbacktransfer}

As in previous sections, 
$F$ is a closed $d$-dimensional oriented PL manifold and
$\xi$ is an oriented PL $F$-block bundle with total space $X$
over a finite ball complex $K,$ $B=|K|.$
We shall geometrically construct a pullback transfer
\[ \xi^!_\PB:  \Omega^\Witt_n (B) \longrightarrow
         \Omega^\Witt_{n+d} (X) \]
on Witt bordism.
Let $f: W \to B$ be a continuous map representing an element $[f]$ of
$\Omega^\Witt_n (B)$.
Choose a PL map $g: W\to B$ homotopic to $f$.
We follow Casson's method for pulling back $F$-block bundles, \cite{casson}.
(Note that the pullback of block and mock bundles
is not generally defined through cartesian diagrams.)
There is a compact polyhedron $V$ and a factorization
\[ \xymatrix{
W \ar@{^{(}->}[r]^j \ar[rd]_g & B \times V \ar[d]^{\pro_1} \\
& B
} \]
of $g$ into a PL embedding $j$ followed by a standard projection.
Let $L$ be a cell complex with $|L|=V$.
The $F$-block bundle pullback $\pro^*_1 \xi$ is by definition
$\xi \times L$,
an $F$-block bundle over the cell complex $K \times L$
with total space
$E(\pro^*_1 \xi) = X\times V.$
Thus the first factor projection 
$X \times V \to X$
defines a PL map
\[ \pro_1: E(\pro^*_1 \xi) \longrightarrow X. \]
Let $C$ be the product cell complex $C := K \times L$ and put 
$\eta := \pro^*_1 \xi.$
Let $C'$ be a subdivision of $C$ such that the subpolyhedron 
$j(W) \subset V\times B$ is given by
$j(W) = |D'|$ for a subcomplex $D'$ of $C'$.
Block bundles can be subdivided and this does not change the total space,
\cite[p. 37]{casson}.
Let $\eta'$ over $C'$ be a subdivision of $\eta$,
$E(\eta') = E(\eta)$.
Block bundles can be restricted to subcomplexes. The total space of the
restriction is given by the union of the blocks over the cells of the
subcomplex. Thus we can restrict $\eta'$ to the subcomplex $D'$ of $C'$
and obtain an $F$-block bundle
$\eta'|_{D'}$
whose total space is a PL subspace
$E(\eta'|_{D'}) \hookrightarrow E(\eta') = E(\eta).$
The composition
\[ E(\eta'|_{D'}) \hookrightarrow E(\eta) \longrightarrow X \] 
gives a map
\begin{equation} \label{equ.eetaprimedprimex}
\overline{g}: E(\eta'|_{D'}) \longrightarrow X. 
\end{equation}
Let $j^* \eta$ be the $F$-block bundle over $W$ that corresponds to
$\eta'|_{D'}$ under the PL homeomorphism $j: W \cong j(W)$.
The pullback $F$-block bundle $g^* \xi$ is then defined to be
\[ g^* \xi = j^* \eta = j^* (\pro^*_1 \xi). \]
The map (\ref{equ.eetaprimedprimex}) is thus a map
\[ \overline{g}: E(g^* \xi) \longrightarrow X. \]
Note that $E(g^* \xi)$ is a compact polyhedron.
In the above construction of pullbacks and $\overline{g}$,
it was not important that the Witt domain $W$ has empty boundary;
everything applies to compact $W$ with boundary as well.
Indeed, Casson's pullback applies of course to PL maps with general
polyhedral domain.
Let $\xi,\xi'$ be $F$-block bundles over cell complexes $K,K'$ such that
$|K|=B=|K'|$.
Recall that $\xi$ and $\xi'$ are called \emph{equivalent} if for some
common subdivision $K''$ of $K$ and $K'$, the
subdivision of $\xi$ over $K''$ is isomorphic to the
subdivision of $\xi'$ over $K''$.
(An \emph{isomorphism} of $F$-block bundles over the same complex is
a block preserving homeomorphism of total spaces.)
An equivalence $\phi: \xi \cong \xi'$ of $F$-block bundles over $B$
induces an equivalence
\[ g^* \phi: g^* \xi \cong g^* \xi'  \]
such that
\[ \xymatrix{
E(g^* \xi) \ar[d]_{g^* \phi}^\cong \ar[r]^{\overline{g}}
   & E(\xi) =X \ar[d]^\phi_\cong \\
E(g^* \xi') \ar[r]^{\overline{g}'} & E(\xi') =X'
} \]
commutes.

\begin{lemma} \label{lem.fiberprodiswitt}
Let $g: W \to B$ be a PL map defined on a compact Witt space with
possibly nonempty boundary $\partial W$.
Then the total space $E(g^* \xi)$ is a closed Witt space with boundary
$E((g^* \xi)|_{\partial W})$.
\end{lemma}
\begin{proof}
An $F$-block bundle is in particular a mock bundle.
Thus $g^* \xi$ is a mock bundle over the Witt space $W$
and the result follows from
Lemma \ref{lem.mockbundleoverwittiswitt}.
\end{proof}
By Lemma \ref{lem.fiberprodiswitt}, 
the map $\overline{g}: E(g^* \xi) \to X$ represents an element
$[\overline{g}] \in \Omega^\Witt_{n+d} (X)$.\\

For future reference and additional clarity in subsequent arguments,
let us record explicitly:
\begin{lemma} \label{lem.homeoyieldswittbord}
Let $W, W'$ be closed $n$-dimensional Witt spaces.
If $f\simeq f':W\to X$
are homotopic maps, then $[f] = [f'] \in \Omega^\Witt_n (X).$
If $\phi: W \stackrel{\cong}{\longrightarrow} W'$ is a PL homeomorphism,
and $f:W\to X,$ $f':W' \to X$ maps such that $f' \circ \phi = f$,
then
$[f] = [f'] \in \Omega^\Witt_n (X).$
\end{lemma}
\begin{proof}
The first statement, asserting homotopy invariance, is part of the
fact that Witt bordism constitutes a homology theory and is proven
by considering a homotopy as a Witt bordism, noting that the cylinder
on a closed Witt space is a Witt space with boundary.
The bordism required by the second statement is given by taking a 
cylinder on the domain of the PL homeomorphism
and a cylinder on the target of the PL homeomorphism, and then gluing the
two cylinders using the homeomorphism.
\end{proof}

\begin{lemma}
The class 
\[ [\overline{g}: E(g^* \xi) \to X] \in \Omega^\Witt_{n+d} (X) \]
depends only on the Witt class $[g] \in \Omega^\Witt_n (B)$.
\end{lemma}
\begin{proof}
Let $g_0: W_0 \to B$ and $g_1: W_1 \to B$ be PL maps such that
$[g_0] = [g_1] \in \Omega^\Witt_n (B)$.
Let $G:W \to B$ be a Witt bordism between $g_0$ and $g_1$;
we may assume $G$ to be PL.
Let $i_j: W_j \hookrightarrow W$ denote the boundary inclusions,
$j=0,1$.
Since $g_j = G \circ i_j$, there is an equivalence
$g^*_j \xi \cong i^*_j G^* \xi$ such that
\[ \xymatrix{
E(g^*_j \xi) \ar[d]_\cong \ar[rd]^{\overline{g}_j} & \\
E(i^*_j G^* \xi) \ar[r]_{\overline{G i_j}} & E(\xi) =X 
} \]
commutes.
Thus $\overline{g}_j$ and $\overline{G i_j}$ are Witt bordant,
$j=0,1$, by Lemma \ref{lem.homeoyieldswittbord}.
According to Lemma \ref{lem.fiberprodiswitt},
$E(G^* \xi)$ is a compact Witt space with boundary
$E(i^*_0 G^* \xi) \sqcup E(i^*_1 G^* \xi)$.
The diagram
\[ \xymatrix{
E(i^*_j G^* \xi) \ar@{^{(}->}[d] \ar[rd]^{\overline{G i_j}} & \\
E(G^* \xi) \ar[r]_{\overline{G}} & E(\xi) =X 
} \]
commutes, $j=0,1$.
Hence, $\overline{G}$ is a Witt bordism between
$\overline{G i_0}$ and $\overline{G i_1}$.
\end{proof}

We define the \emph{geometric transfer} (or \emph{pullback transfer})
\[ \xi^!_\PB:  \Omega^\Witt_n (B) \longrightarrow
         \Omega^\Witt_{n+d} (X) \]
by
\[ \xi^!_\PB [g:W\to B] = [\overline{g}: E(g^* \xi) \to X], \]
where $g$ is a PL representative of the bordism class.
Let
\[ \xi^!_{BRS}: \Omega^\Witt_n (B) \longrightarrow
         \Omega^\Witt_{n+d} (X)  \]
be the map
\[ \xi^!_{BRS} [g] := \rho_* (u_{BRS} (\nu) \cap T(\xi)_* \sigma [g]),  \]
where $\nu = \nu_\xi: X\to \BBSPL_{s-d}$ represents the stable
vertical normal PL disc block bundle of $\xi$.
This is a technical intermediary; in terms of their respective definitions,
the difference between $\xi^!_{BRS}$ and $\xi^!$ is that the former uses
the Thom class $u_{BRS} (\nu)$, while the latter uses the Thom class $u_\SPL (\mu)$.
We will eventually see that $\xi^!_\PB = \xi^!_{BRS} = \xi^!$
on Witt bordism. Towards that goal, let us first investigate the behavior 
of both the pullback transfer and the BRS-transfer 
under standard factor projections.

\begin{prop} \label{prop.transferandprojection}
Let $B$ and $D$ be compact polyhedra.
Let $\xi \times D$ denote the $F$-block bundle over $B\times D$ 
obtained by pulling
back $\xi$ under the projection $\pro_1: B\times D \to B$.
Then the diagrams
\begin{equation} \label{equ.pbtransferproj}
\xymatrix@C=60pt{
\Omega^\Witt_n (B) \ar[r]^{\xi^!_\PB}
  & \Omega^\Witt_{n+d} (X) \\
\Omega^\Witt_n (B \times D) \ar[u]^{\pro_{1*}} \ar[r]^{(\xi \times D)^!_\PB}
  & \Omega^\Witt_{n+d} (X \times D) \ar[u]_{\pro_{1*}}  
} \end{equation}
and
\begin{equation} \label{equ.brstransferproj}
\xymatrix@C=60pt{
\Omega^\Witt_n (B) \ar[r]^{\xi^!_{BRS}}
  & \Omega^\Witt_{n+d} (X) \\
\Omega^\Witt_n (B \times D) \ar[u]^{\pro_{1*}} \ar[r]^{(\xi \times D)^!_{BRS}}
  & \Omega^\Witt_{n+d} (X \times D) \ar[u]_{\pro_{1*}}  
} \end{equation}
commute.
\end{prop}
\begin{proof}
We will first establish the commutativity of
diagram (\ref{equ.pbtransferproj}) involving the pullback
transfers.
Recall that the $F$-block bundle $\xi$ is given over a
cell complex $K$ with $|K|=B$.
Let $J$ be a cell complex with polyhedron $|J|=D$.
Then $\xi \times D$ is an $F$-block bundle over the
cell complex $K\times J$.
Let $[g] \in \Omega^\Witt_n (B \times D)$ be any element,
represented by a PL map $g: W \to B \times D$.
Choose a compact polyhedron $V$ and a factorization of $g$ as
\begin{equation} \label{equ.factorizwtobd}
\xymatrix@C=50pt{
W \ar@{^{(}->}[r]^j \ar[rd]_g 
   & (B \times D) \times V \ar[d]^{\pro_{B\times D}} \\
& B \times D
} \end{equation}
Let $L$ be a cell complex with $|L|=V$.
We will compute $\xi^!_\PB \pro_{1*} [g]$.
The element $\pro_{1*} [g]$ is represented by $\pro_1 \circ g$
with factorization
\[ \xymatrix@C=50pt{
W \ar@{^{(}->}[r]^j \ar[rd]_{\pro_1 \circ g} 
   & B \times D \times V \ar[d]^{\pro_B} \\
& B.
} \]
The pullback
$\pro^*_B \xi = \xi \times J \times L$
has total space $E(\pro^*_B \xi) = X \times D \times V$
which projects to $X$ via
\[ \pro_X: E(\pro^*_B \xi) = X \times D \times V
  \longrightarrow X = E(\xi). \]
Let $C$ be the cell complex $C=K \times J \times L$ and
let $C'$ be a subdivision of $C$ such that
$j(W)$ is given by $j(W) = |D'|$
for some subcomplex $D'$ of $C'$.
Let $(\pro^*_B \xi)'$ be the block bundle over $C'$ obtained by
subdivision of $\pro^*_B \xi$.
Then $(\pro^*_B \xi)'$ can be restricted to $D'$, and the total
space of this restriction $(\pro^*_B \xi)'|_{D'}$ is a subspace
of $E((\pro^*_B \xi)')=E(\pro^*_B \xi)$.
The composition of the subspace inclusion with $\pro_X$ yields
a map  
\[ 
\overline{\pro_1 \circ g}:
  E((\pro^*_B \xi)'|_{D'}) \subset
  E(\pro^*_B \xi) = X \times D \times V
  \stackrel{\pro_X}{\longrightarrow} X
\]
such that
\[ \xi^!_\PB [\pro_1 \circ g]
    = [\overline{\pro_1 \circ g}].  \]

Let us compute $(\xi \times D)^!_\PB [g]$.
The relevant factorization is (\ref{equ.factorizwtobd});
the pullback
\[ \pro^*_{B \times D} (\xi \times D) 
   = (\xi \times D) \times L = \xi \times J \times L
     = \pro^*_B \xi \]
has total space 
$E(\pro^*_{B \times D} (\xi \times D)) = X \times D \times V$
which projects to $X \times D$ via
\[ \pro_{X \times D}: E(\pro^*_{B \times D} (\xi \times D)) 
    = (X \times D) \times V
  \longrightarrow X \times D = E(\xi \times D). \]
Let $(\pro^*_{B \times D} (\xi \times D))'$ 
be the block bundle over $C'$ obtained by
subdivision of $\pro^*_{B \times D} (\xi \times D)$.
Then
$(\pro^*_{B \times D} (\xi \times D))' = (\pro^*_B \xi)'$
and thus
\[ (\pro^*_{B \times D} (\xi \times D))'|_{D'} =
      (\pro^*_B \xi)'|_{D'}.  \]
Consider the commutative diagram
\[ \xymatrix@C=60pt{
E((\pro^*_{B \times D} (\xi \times D))'|_{D'}) \ar@{^{(}->}[r] \ar@{=}[d] &
  E(\pro^*_{B \times D} (\xi \times D)) \ar[r]^{\pro_{X \times D}} \ar@{=}[d] &
    X \times D \ar[d]^{\pro_1} \\
E((\pro^*_B \xi)'|_{D'}) \ar@{^{(}->}[r] &
  E(\pro^*_B \xi) \ar[r]^{\pro_X} &
    X.
} \]
The upper horizontal composition is a map $\overline{g}$ such that
\[ (\xi \times D)^!_\PB [g] =
       [\overline{g}],  \]
and the lower horizontal composition is
$\overline{\pro_1 \circ g}$.       
Therefore,
\begin{align*}
\pro_{1*} (\xi \times D)^!_\PB [g]
&= [\pro_1 \circ \overline{g}] = [\overline{\pro_1 \circ g}] \\
&= \xi^!_\PB [\pro_1 \circ g] = \xi^!_\PB \pro_{1*} [g].
\end{align*}
Thus, (\ref{equ.pbtransferproj}) commutes as claimed.

It remains to establish the commutativitiy of
diagram (\ref{equ.brstransferproj}).
Let $\nu_\xi = \nu_\theta$ denote the stable vertical normal 
PL disc block bundle
associated to a particular choice of blockwise embedding
$\theta: X=E(\xi) \hookrightarrow \real^s \times B$
by Proposition \ref{thm.existencenormaldiscblockbundle}.
The PL embedding
\[ E(\xi \times D) = X \times D
   \stackrel{\theta \times \id_D}{\hookrightarrow}
     (\real^s \times B) \times D = \real^s \times (B \times D) \]
is block preserving over $B\times D$
with respect to the $F$-blocks
$(\xi \times D)(\sigma \times \tau) = \xi (\sigma) \times \tau$
of $\xi \times D$,
$\sigma \in K,$ $\tau \in J$, as
\[  
(\theta \times \id_D)(\xi (\sigma) \times \tau)
= \theta (\xi (\sigma)) \times \tau
\subset (\real^s \times \sigma) \times \tau
= \real^s \times (\sigma \times \tau).
\]
Thus the stable vertical normal disc-block bundle
of $\xi \times D$
can be computed from the embedding $\theta \times \id_D$,
which yields
\[ \nu_{\xi \times D} = \nu_{\theta \times \id_D} =
   \nu_\theta \times D = \nu_\xi \times D, \]
a disc block bundle over $X \times D$.   
Recall that the Thom space of the block bundle $\nu_\xi$ is
\[ \Th (\nu_\xi)
 = N \cup_{\partial N} \cone (\partial N) \]
with $N=E(\nu_\xi)$.
Thus, with 
$N' := E(\nu_{\xi \times D}) = N \times D,$
we have
\[ \Th (\nu_{\xi \times D})
 = N' \cup_{\partial N'} \cone (\partial N').   \]
Here, $\partial N'$ denotes the total space of the sphere bundle
of $\xi \times D,$
$\partial N' = (\partial N)\times D.$
The projection
$\pro_1: N' = N \times D \to N$
induces a map
\[ \Th(\pro_1): (N \times D) \cup_{(\partial N) \times D}
                \cone ((\partial N) \times D)
    \longrightarrow 
    N \cup_{\partial N} \cone (\partial N),
\]
i.e. a map
\[ \Th (\pro_1):
  \Th (\nu_{\xi \times D}) \longrightarrow
  \Th (\nu_\xi). \]
The suspension of the projection $\pro_1: B \times D \to B$ 
is a map
$S^s \pro_1: S^s (B\times D)^+ \to S^s B^+.$
The $F$-block bundle $\xi \times D$ has its Umkehr map
\[ T(\xi \times D):
      S^s (B\times D)^+ \longrightarrow
        \Th (\nu_{\xi \times D}) \]
such that the diagram        
\[ \xymatrix@C=60pt{
S^s B^+ \ar[r]^{T(\xi)} &
  \Th (\nu_\xi) \\
S^s (B\times D)^+ \ar[u]^{S^s \pro_1} \ar[r]^{T(\xi \times D)} &
  \Th (\nu_{\xi \times D}) \ar[u]_{\Th (\pro_1)}
} \]
commutes up to homotopy. The induced diagram
\[ \xymatrix@C=60pt{
\widetilde{\Omega}^\Witt_{n+s} (S^s B^+) \ar[r]^{T(\xi)_*} &
  \widetilde{\Omega}^\Witt_{n+s} (\Th \nu_\xi) \\
\widetilde{\Omega}^\Witt_{n+s} (S^s (B\times D)^+) 
   \ar[u]^{(S^s \pro_1)_*} \ar[r]^{T(\xi \times D)_*} &
  \widetilde{\Omega}^\Witt_{n+s} (\Th \nu_{\xi \times D}) 
    \ar[u]_{\Th (\pro_1)_*}
} \]
on reduced Witt bordism commutes.
The diagram
\[ \xymatrix@C=60pt{
\Omega^\Witt_n (B) \ar[r]^{\sigma}_\cong &
  \widetilde{\Omega}^\Witt_{n+s} (S^s B^+) \\
\Omega^\Witt_n (B\times D) 
   \ar[u]^{\pro_{1*}} \ar[r]^{\sigma}_\cong &
  \widetilde{\Omega}^\Witt_{n+s} (S^s (B\times D)^+) \ar[u]_{(S^s \pro_1)_*}
} \]
commutes by the naturality of the suspension isomorphism $\sigma$.

It remains to show that
\begin{equation} \label{equ.capuandrhoandproj} 
\xymatrix@C=60pt{
\widetilde{\Omega}^\Witt_{n+s} (\Th \nu_\xi) 
 \ar[r]^{u_{BRS} (\nu_\xi) \cap -} &
  \Omega^\Witt_{n+d} (E\nu_\xi)
   \ar[r]^{\rho_*}_\cong &
    \Omega^\Witt_{n+d} (X) \\
\widetilde{\Omega}^\Witt_{n+s} (\Th \nu_{\xi \times D})
 \ar[u]^{\Th (\pro_1)_*} 
 \ar[r]^{u_{BRS} (\nu_{\xi \times D}) \cap -} &
  \Omega^\Witt_{n+d} (E\nu_{\xi \times D})
   \ar[u]^{\pro_{1*}}
   \ar[r]^{\rho_*}_\cong &
    \Omega^\Witt_{n+d} (X \times D) 
     \ar[u]_{\pro_{1*}}
} \end{equation}
commutes.
The right hand side commutes, since
the zero section embedding
$X \times D \hookrightarrow E \nu_{\xi \times D} = N \times D$
of $\nu_{\xi \times D}$ is given by
$i \times \id_D$, where $i$ is the zero section embedding
$i: X \hookrightarrow E\nu_\xi =N$ of $\nu_\xi$, 
so that
\[ \xymatrix@C=60pt{
N = E\nu_\xi & X \ar[l]_i \\
N\times D = E\nu_{\xi \times D} \ar[u]_{\pro_1} &
  X \times D \ar[l]^{i \times \id_D} \ar[u]^{\pro_1}
} \]
commutes.
We will prove that the left hand side commutes as well.
The map
\[ \Th (\pro_1):
  \Th (\nu_{\xi \times D}) \longrightarrow
  \Th (\nu_\xi) \]
induces a homomorphism
\[ \Th (\pro_1)^*:
 \Omega^{d-s}_\SPL (E\nu_\xi, \dot{E} \nu_\xi)
 \longrightarrow
 \Omega^{d-s}_\SPL (E\nu_{\xi \times D}, \dot{E} \nu_{\xi \times D}), \]
which agrees with the homomorphism
\[ \pro^*_1: 
 \Omega^{d-s}_\SPL (N, \partial N)
 \longrightarrow
 \Omega^{d-s}_\SPL ((N, \partial N) \times D) \]
induced by the map of pairs
$\pro_1: (N, \partial N) \times D \to (N, \partial N).$
By the naturality of the BRS-Thom class
(\cite[top of p. 27]{buonrs}),
this homomorphism maps the BRS-Thom class
of $\nu_\xi$ to the BRS-Thom class of 
$\pro^*_1 \nu_\xi = \nu_{\xi \times D}$,
\[ \pro^*_1 (u_{BRS} (\nu_\xi)) = u_{BRS} (\nu_{\xi \times D}).  \] 
Given any element 
\[ [g] \in  \widetilde{\Omega}^\Witt_{n+s} (\Th \nu_{\xi \times D})
   = \Omega^\Witt_{n+s} ((N,\partial N)\times D),
\]
the computation 
\begin{align*}
u_{BRS} (\nu_\xi) \cap \pro_{1*} [g]
&= \pro_{1*} (\pro^*_1 u_{BRS} (\nu_\xi) \cap [g]) \\
&= \pro_{1*} (u_{BRS} (\nu_{\xi \times D}) \cap [g])
\end{align*}
shows that the left hand side of diagram 
(\ref{equ.capuandrhoandproj}) commutes.
\end{proof}

The pullback transfer $\xi^!_\PB$ on Witt bordism
agrees with the transfer $\xi^!_{BRS}$:
\begin{prop} \label{prop.geopullbtransfistpcapu}
The diagram
\[ \xymatrix@C=50pt{
\widetilde{\Omega}^\Witt_{n+s} (S^s B^+)
 \ar[r]^{T(\xi)_*} &
 \widetilde{\Omega}^\Witt_{n+s} (\Th (\nu))
 \ar[r]^{u_{BRS} (\nu) \cap -} &
 \Omega^\Witt_{n+d} (E(\nu)) 
  \ar[d]^{\rho_*}_\cong \\
\Omega^\Witt_n (B) \ar[u]^\sigma_\cong
  \ar[rr]^{\xi^!_\PB} & &
   \Omega^\Witt_{n+d} (X)
} \]
commutes, that is, $\xi^!_\PB = \xi^!_{BRS}$.
\end{prop}
\begin{proof}
Let $h: W^n \to B$ be a continuous map from a closed $n$-dimensional
Witt space $W$ to $B$, representing an element
$[h]\in \Omega^\Witt_n (B)$. By simplicial approximation, we may
assume that $h$ is PL.
We begin by observing that by Proposition
\ref{prop.transferandprojection}, it suffices to prove
the statement for the case where $h: W\to B$ is a PL embedding:
For given any PL map $h: W\to B$,
consider the graph embedding
\[ (h,\id_W): W \longrightarrow B \times W \]
which factors $h$ as
\[ \xymatrix@C=50pt{
W \ar@{^{(}->}[r]^{(h,\id_W)} \ar[rd]_h & B \times W
   \ar[d]^{\pro_1} \\
& B.   
} \]
Let $\xi \times W$ denote the $F$-block bundle over $B\times W$ 
obtained by pulling
back $\xi$ under the projection $\pro_1: B\times W \to B$.
If the statement is known for embeddings, then
\[  (\xi \times W)^!_\PB [(h,\id_W)] 
         = (\xi \times W)^!_{BRS} [(h,\id_W)]. \]
Hence by Proposition \ref{prop.transferandprojection} with $D=W$,
\begin{align*}
\xi^!_\PB [h]
&= \xi^!_\PB [\pro_1 \circ (h,\id_W)] 
  = \xi^!_\PB \pro_{1*} [(h,\id_W)] \\
&= \pro_{1*} (\xi \times W)^!_\PB [(h,\id_W)] 
  = \pro_{1*} (\xi \times W)^!_{BRS} [(h,\id_W)] \\
&= \xi^!_{BRS} \pro_{1*} [(h,\id_W)] 
  = \xi^!_{BRS} [h].
\end{align*}
Consequently, it remains to prove the equality 
$\xi^!_\PB = \xi^!_{BRS}$ on Witt bordism classes that are
represented by PL embeddings.

As in the construction of the Umkehr map $T(\xi)$ in 
Section \ref{sec.umkehrmap}, let $N$ denote the total space $E(\nu)$ of the
stable vertical normal closed disc block bundle $\nu = \nu_\theta$
of $\xi$ associated to a choice of block preserving embedding 
$\theta: X\hookrightarrow \real^s \times B$, where $X$ is the total
space of the given $F$-block bundle $\xi$.
Thus $N$ is a $\xi$-block preserving
regular neighborhood of $\theta (X)$ in $\real^s \times B$.
Recall that $\partial N$ denotes the total space of the sphere
block bundle of $\nu$.
Let $D^s \subset \real^s$ be a closed PL ball about the origin which is
large enough so that
$(D^s - \partial D^s) \times B$ contains 
$N \cup V \subset \real^s \times B$, where $V$ is the
outside collar to $\partial N$ used in the construction of the
Umkehr map; such a ball exists by compactness of $X$.

Let $h: W \hookrightarrow B$ be a PL embedding of a closed Witt space into $B$.
Recall that $K$ is a cell complex with polyhedron $|K|=B$ and $\xi$
is given over $K$.
By subdivision of $K$ and $\xi$, we may assume that
$h(W)=|K_W|$ for a subcomplex $K_W$ of $K$.
Let $L_S$ be a finite simplicial complex such that
\begin{enumerate}
\item[(1)] $|L_S| = S^s B^+$,
\item[(2)] there is a subcomplex $L$ of $L_S$ such that $|L| = D^s \times B$,
\item[(3)] for every simplex $\sigma \in K$, there is a subcomplex
  $L_\sigma$ of $L$ such that
  \[ |L_\sigma| = D^s \times \sigma, \]
\item[(4)] there exists a subcomplex $L_\theta$ of $L$ 
           such that $|L_\theta| = \theta (X)$,
\item[(5)] the stable vertical normal bundle $\nu$ is a (disc-) block bundle over
  the complex $L_\theta$ such that
  \[ E(\nu) \cap (D^s \times \sigma) =
             \bigcup_{\tau \in L_\sigma \cap L_\theta} \nu (\tau), \]
  where $\nu (\tau)$ is the disc-block of $\nu$ over the simplex $\tau$.             
\end{enumerate}
Property (3) implies that $L_\tau$ is a subcomplex
of $L_\sigma$ for every face $\tau$ of $\sigma \in K$.
Furthermore,
\[
D^s \times h(W) = D^s \times |K_W| = \bigcup_{\sigma \in K_W} D^s \times \sigma
  = \bigcup_{\sigma \in K_W} |L_\sigma| 
  = | \bigcup_{\sigma \in K_W} L_\sigma |
\]
so that
\[ L_W := \bigcup_{\sigma \in K_W} L_\sigma \]
is a simplicial subcomplex of $L$ with
$D^s \times h(W) = |L_W|.$
Since the embedding $\theta: X \hookrightarrow \real^s \times B$ is
block preserving with respect to the $F$-blocks of $\xi$, we have
$\theta (\xi (\sigma)) = (D^s \times \sigma) \cap \theta (X)$
for all $\sigma \in K$.
Hence by property (4),
\[
\theta (\xi (\sigma)) =
 |L_\sigma| \cap |L_\theta| = |L_\sigma \cap L_\theta|.
\]
Thus the embedded $F$-blocks $\theta (\xi (\sigma))$ are triangulated by the
subcomplex $L_\sigma \cap L_\theta$ of $L$.

The image $\sigma [h]$ under the suspension isomorphism
\[ \sigma: \Omega^\Witt_n (B) \stackrel{\cong}{\longrightarrow} 
   \widetilde{\Omega}^\Witt_{n+s} (S^s B^+)
   = \Omega^\Witt_{n+s} ((D^s, \partial D^s) \times B) \]
is represented by the closed product PL embedding
\[ \id \times h: (D^s \times W, \partial (D^s \times W)) \hookrightarrow
       (D^s \times B, (\partial D^s) \times B). \]
The Umkehr map is a PL map
\[ T(\xi): S^s B^+ = \Th (\real^s \times B) 
   = \frac{D^s \times B}{(\partial D^s) \times B} 
       \longrightarrow \Th (\nu) \]
which is the identity near $\theta (X)$.
Composing with it, we obtain a PL map
\[ f = T(\xi) \circ (\id \times h):
  (D^s \times W, \partial (D^s \times W)) \longrightarrow
  (\Th (\nu), \infty). \]
Let $A$ be the ball complex with
$|A| = N= E(\nu)$ whose balls include the blocks of $\nu$.
The rest of the balls come from the sphere block bundle of $\nu$.
The BRS Thom class
$u_{BRS} (\nu) \in \Omega^{-(s-d)}_\SPL (N, \partial N)$
is represented by the mock bundle $\eta$ with projection given by the
zero section 
$i: \theta (X) \to A$. 
Thus the total space of $\eta$ is $E(\eta)=\theta (X)$.
The mock bundle $\eta$ is an embedded mock bundle in the sense
of Buoncristiano-Rourke-Sanderson \cite[p. 34]{buonrs}:
The restriction $i|: \eta (\sigma) \to \sigma$ for a ball 
$\sigma = \nu (\tau) \in A$ is the inclusion $\tau \hookrightarrow \nu (\tau)$,
which is locally flat by definition of a disc block bundle.
Furthermore, $i|: \eta (\sigma) \to \sigma$ is proper,
i.e. 
$i|^{-1} (\partial \sigma) = \partial \eta (\sigma).$
We wish to compute the cap product
\[ u_{BRS} (\nu) \cap [f: (D^s,\partial D^s) \times W 
   \to (\Th (\nu),\infty)] \in 
    \Omega^\Witt_{n+d} (E(\nu)).  \]
The base complex of $\eta$ is only known to be a ball complex, 
not a simplicial complex
as required for pulling back a mock bundle via a cartesian square.
Thus we need to subdivide simplicially.
Let $L'$ be a simplicial subdivision of $L$ and 
let $A'$ be a simplicial subdivision of $A$ such that
$A'$ is a subcomplex of $L'$. Thus,
\[ |A'| = E(\nu),~ |L'| = D^s \times B. \]
The complex $L'$ contains a (simplicial) subcomplex $L'_W$ given by
\[ L'_W = \{ \tau \in L' ~|~ \tau \subset \sigma
    \text{ for some } \sigma \in L_W \}.  \]
This is a subdivision of $L_W \subset L$, and
\[ |L'_W| = |L_W| = D^s \times h(W). \]
So the inclusion
\[ |L'_W| = D^s \times h(W) \hookrightarrow D^s \times B = |L'| \]
is a \emph{simplicial} map
\[ L'_W \hookrightarrow L'. \]
By \cite[Thm. 2.1, p. 23]{buonrs} (see also 
\cite[Subdivision theorem, p. 128]{rourke})
mock bundles can be subdivided:
If $\alpha$ is a mock bundle over a ball complex $D$ 
with total space $E(\alpha)$ and projection $p: E(\alpha) \to D$,
and $D'$ is a 
subdivision of $D$, then there exists a mock bundle $\alpha'$ over
$D'$ together with a PL homeomorphism
$\phi: E(\alpha) \stackrel{\cong}{\longrightarrow} 
  E(\alpha')$
which preserves $\alpha$-blocks over $D$, and a homotopy
\[ F: E(\alpha) \times I \longrightarrow |D|=|D'|,~
   F_0 =p,~ F_1 = p' \phi, \]
which respects the $\alpha$-blocks over $D$. 
(Here, $p': E(\alpha')\to |D'|$
is the projection of $\alpha'$.)
Moreover, if $\alpha$ is an embedded mock bundle,
then the subdivision theorem yields
again an embedded mock bundle and the homotopy can be taken to be
an isotopy which is covered by an ambient isotopy.
We apply this to the zero section mock bundle $\eta$ over $A$:
Since $A'$ is a (simplicial) subdivision of $A$, there thus
exists a correspondingly subdivided mock bundle
$\eta'$ over $A'$.
Since $\eta$ is an embedded mock bundle
$i: \theta (X) = E(\eta) \hookrightarrow E(\nu),$
so is $\eta'$.
Thus the projection map $i'$ of $\eta'$ may be taken to be a PL embedding
$i': E(\eta') \hookrightarrow E(\nu).$
As the zero section $i$ does not touch the sphere bundle of $\nu$
(i.e. $\eta$ has empty blocks over $\partial N$), the same
is true for the perturbation $i'$.
There exists a PL homeomorphism
$\phi: \theta (X) = E(\eta) \stackrel{\cong}{\longrightarrow} 
  E(\eta')$
which preserves $\eta$-blocks over the ball complex $A$.
The maps $i$ and $i' \phi$ are isotopic via an isotopy
\[ F: \theta (X) \times I \longrightarrow E(\nu) \times I,~
   F_0 =i,~ F_1 = i' \phi. \]
This isotopy is covered by an ambient isotopy
\[ H: E(\nu) \times I \longrightarrow E(\nu) \times I,~ H_0 = \id, \]
so that
\[ \xymatrix{
\theta (X) \times I \ar[rd]_F \ar[rr]^{F_0 \times \id = i\times \id} & &
  E(\nu) \times I \ar[ld]^H \\
& E(\nu) \times I &
} \]
commutes. This implies
\begin{equation} \label{equ.h1iisiprimephi}
H_1 \circ i = F_1 = i' \circ \phi. 
\end{equation}
By an induction on the cells $\sigma \in K$, starting with
the $0$-dimensional cells, $F$ and $H$ can be constructed to
preserve blocks over $K$. More precisely:
Let $\nu_\sigma$ denote the restriction of $\nu$ to the 
embedded $F$-block
$\theta (\xi (\sigma)) = \theta (X) \cap (D^s \times \sigma).$
Since $\nu$ is a block bundle over the complex $L_\theta$
and $\theta (\xi (\sigma))$ is triangulated by 
$L_\sigma \cap L_\theta$, the total space of $\nu_\sigma$
is given by $E(\nu_\sigma) = \bigcup_\tau \nu (\tau)$,
where $\tau$ ranges over all simplices of $L_\sigma \cap L_\theta$.
Thus by property (5) above,
\begin{equation} \label{equ.enusigmaenucapsigmads}
E(\nu_\sigma) = E(\nu) \cap (D^s \times \sigma). 
\end{equation}
Then $H$ can be inductively arranged to satisfy
\begin{equation} \label{equ.htenusigmaisenusigma}
  H_t (E(\nu_\sigma)) = E(\nu_\sigma)
\end{equation}
for all $\sigma \in K$ and all $t\in [0,1]$, as follows:
Recall that Buoncristiano, Rourke and Sanderson's construction of $H$ in
their proof of the
mock bundle subdivision theorem proceeds inductively 
over cells of the base, starting with the 0-cells.
In the present context, one organizes their induction as follows:
Start with the $0$-skeleton $A^0$ of $A$. 
For every $0$-cell $\sigma^0$ of $K$, subdivide $\eta$ over
$A^0 \cap D^s \times \sigma^0$ within the manifold 
$E(\nu) \cap D^s \times \sigma^0$.
Extend this subdivision for every $1$-cell $\sigma^1$ of $K$ to
a subdivision over $A^0 \cap D^s \times \sigma^1$
within the manifold $E(\nu) \cap D^s \times \sigma^1$. 
Continue in this way with
$2$-cells $\sigma^2$, etc., until all cells of $K$ have been used.
Then move on to the $1$-skeleton $A^1$ of $A$:
For every $0$-cell $\sigma^0$ of $K$,
extend the subdivision to
a subdivision over $A^1 \cap D^s \times \sigma^0$
within the manifold $E(\nu) \cap D^s \times \sigma^0$.
Extend this subdivision for every $1$-cell $\sigma^1$ of $K$ to
a subdivision over $A^1 \cap D^s \times \sigma^1$
within the manifold $E(\nu) \cap D^s \times \sigma^1$, and so on.

The mock bundle $\eta'$ is defined over the simplicial complex
$A'$ with polyhedron $|A'| = E(\nu)$, but
using the canonical inclusions $E(\nu) \subset D^s \times B$
and $E(\nu) \subset \Th (\nu),$ we may regard $\eta'$
as a mock bundle over $D^s \times B$, and as a mock bundle over
$\Th (\nu)$. In more detail, the composition
\[ E(\eta') \stackrel{i'}{\hookrightarrow}
   E(\nu) = |A'| \hookrightarrow D^s \times B =|L'| \]
is the projection of a mock bundle over the complex $L'$,
whose blocks over simplices in $A'$ are the blocks of $\eta'$
and blocks over simplices not in $A'$ are taken to be empty.
(Here, we are using that $\eta'$ has empty blocks over the sphere
bundle $\partial N$.)
Similarly, after extending the triangulation $A'$ to a
triangulation $T'$ of $\Th (\nu)$ by coning off simplices of $A'$
that are in $\partial N$ (and adding the cone point $\infty$ as
a $0$-simplex), the composition
\[ E(\eta') \stackrel{i'}{\hookrightarrow}
   E(\nu) = |A'| \hookrightarrow \Th (\nu) =|T'| \]
is the projection of a mock bundle over the complex $T'$,
whose blocks over simplices in $A'$ are the blocks of $\eta'$
and blocks over simplices not in $A'$ are again taken to be empty.
In view of the commutative diagram
\[ \xymatrix{
E(\eta') \ar@{^{(}->}[d]_{i'} & \\
E(\nu) \ar@{^{(}->}[d] \ar@{^{(}->}[rd] & \\
D^s \times B \ar[r]_{T(\xi)} & \Th (\nu)
} \]
the pullback $T(\xi)^* (\eta'/_{T'})$ under the Umkehr map is
precisely $\eta'/_{L'}$. Therefore, the mock bundle pullback
$f^* (\eta')$ is given by
\[ f^* (\eta') = (\id \times h)^* T(\xi)^* (\eta'/_{T'}) 
               = (\id \times h)^* (\eta'/_{L'}). \]
The mock bundle $\eta'$ (contrary to $\eta$, possibly)
is defined over a \emph{simplicial} complex
$L'$ and, as pointed out above, the inclusion
$D^s \times h(W) \hookrightarrow D^s \times B = |L'|$
is a \emph{simplicial} map
\[ L'_W \hookrightarrow L'. \]
Therefore, the mock bundle pullback
$f^* (\eta') = (\id \times h)^* (\eta')$
is given by the \emph{cartesian} diagram
\[ \xymatrix{
E((\id \times h)^* \eta') \ar[r] \ar[d] \ar[rd]^g 
  & E(\eta') \ar@{^{(}->}[d]^{i'} \\
D^s \times h(W) =|L'_W| \ar@{^{(}->}[r] & D^s \times B = |L'|.
} \]
It follows that the cap product of the BRS Thom class 
with $[f]$ is given by the diagonal arrow
\[ u_{BRS} (\nu) \cap [f] = [g] \in 
    \Omega^\Witt_{n+d} (E(\nu)),  \]
the total space of the pullback is given by
\[ E((\id \times h)^* \eta')
    = (D^s \times h(W)) \cap i' E(\eta')  \]
and $g$ is the subspace inclusion
\[ g: (D^s \times h(W)) \cap i' E(\eta') \subset i' E(\eta') \subset 
  E(\nu). \]
We shall show next that the final stage $H_1: E(\nu) \to E(\nu)$ 
of the ambient isotopy $H$
induces a homeomorphism
\begin{equation} \label{equ.homeoh1onexikw}
H_1: E(\xi|_{K_W}) \stackrel{\cong}{\longrightarrow}  
      (D^s \times h(W)) \cap i' E(\eta'), 
\end{equation}      
where we use $\theta$ to identify $X=E(\xi)$ and
$\theta (X)$, and to identify
$E(\xi|_{K_W})$ and $\theta (X) \cap (D^s \times |K_W|)$.      
The homeomorphism $H_1$ restricts to a homeomorphism
\[
H_1|: \theta (X) \cap (D^s \times |K_W|)
 \stackrel{\cong}{\longrightarrow}
 H_1 (\theta (X) \cap (D^s \times |K_W|)),
\]
whose target we shall now compute:
\begin{align*}
H_1 (\theta (X) \cap (D^s \times |K_W|))
&= H_1 (\theta (X) \cap E(\nu) \cap (D^s \times |K_W|)) \\
&= H_1 (\theta (X)) \cap H_1 (E(\nu) \cap D^s \times |K_W|) \\
&= H_1 (\theta (X)) \cap 
     H_1 \left( E(\nu) \cap \bigcup_{\sigma \in K_W} D^s \times \sigma \right) \\
&= H_1 (\theta (X)) \cap 
     H_1 \left( \bigcup_{\sigma \in K_W} E(\nu) \cap (D^s \times \sigma) \right) \\
&= H_1 (\theta (X)) \cap 
     \bigcup_{\sigma \in K_W} H_1 (E(\nu) \cap (D^s \times \sigma)) \\
&= H_1 (\theta (X)) \cap 
     \bigcup_{\sigma \in K_W} (E(\nu) \cap (D^s \times \sigma))  
        \text{ (by (\ref{equ.enusigmaenucapsigmads}) 
                and (\ref{equ.htenusigmaisenusigma}))}  \\ 
&= H_1 (\theta (X)) \cap E(\nu) \cap
     \bigcup_{\sigma \in K_W} (D^s \times \sigma) \\
&= H_1 (\theta (X)) \cap
     \bigcup_{\sigma \in K_W} (D^s \times \sigma)  \\
&= H_1 i E(\eta) \cap
     \bigcup_{\sigma \in K_W} |L_\sigma|   \\     
&= i' \phi E(\eta) \cap |L_W| 
     \text{ (by (\ref{equ.h1iisiprimephi}))} \\
&= i' E(\eta') \cap (D^s \times h(W)).             
\end{align*}
Thus we obtain the homeomorphism
(\ref{equ.homeoh1onexikw}).
In the diagram
\[ \xymatrix{
E(\xi|_{K_W}) \ar[d]_{H_1|}^\cong
 \ar@{^{(}->}[r] & E(\xi)=\theta (X) \ar@{^{(}->}[r]^i 
  & E(\nu) \ar[d]^{H_1}_\cong \\
i' E(\eta') \cap (D^s \times h(W)) \ar@{^{(}->}[rr]^g
  & & E(\nu),  
} \]
all the horizontal arrows are subspace inclusions and thus the diagramm commutes.
By Lemma \ref{lem.homeoyieldswittbord} applied to the 
PL homeomorphism $H_1|$,
\[ [g] = [g \circ H_1|] \in \Omega^\Witt_{n+d} (E(\nu)). \]
By commutativity of the diagram,
\begin{align*} 
[g \circ H_1|] 
&= [E(\xi|_{K_W}) \subset \theta (X) 
     \stackrel{H_1 i}{\longrightarrow} E(\nu)] \\
&= [E(\xi|_{K_W}) \subset \theta (X) 
     \stackrel{i' \phi}{\longrightarrow} E(\nu)]     
\end{align*}
By restriction, the isotopy $F$ gives rise to an isotopy
\[ \hat{F}: E(\xi|_{K_W}) \times I \subset
     \theta (X) \times I \stackrel{F}{\longrightarrow} 
     E(\nu) \]
from
\[ \hat{F}_0 =
   E(\xi|_{K_W}) \subset \theta (X) 
     \stackrel{F_0 = i}{\longrightarrow} E(\nu)
\]
to
\[ \hat{F}_1 =
   E(\xi|_{K_W}) \subset \theta (X) 
     \stackrel{F_1 = i' \phi}{\longrightarrow} E(\nu).
\]
By Lemma \ref{lem.homeoyieldswittbord},
\[ [\hat{F}_0] = [\hat{F}_1] \in \Omega^\Witt_{n+d} (E(\nu)). \]
Therefore,
\[ [g] = [g \circ H_1|] = [\hat{F}_1] = [\hat{F}_0]
    \in \Omega^\Witt_{n+d} (E(\nu)). \]

Now the geometric pullback transfer of $[h: W \hookrightarrow B]$
is given by
\[ \xi^!_\PB [h:W\hookrightarrow B] = 
  [E(\xi|_{K_W}) \subset E(\xi) = \theta (X)]. \]
Hence
\[ i_* \xi^!_\PB [h:W\hookrightarrow B] = [\hat{F}_0]. \]
Finally, since $i_*$ and $\rho_*$ are inverses of each other,
\begin{align*}
\xi^!_\PB [h:W\hookrightarrow B]
&= \rho_* [\hat{F}_0] = \rho_* [g] 
    = \rho_* (u_{BRS} (\nu) \cap [f]) \\
&= \rho_* (u_{BRS} (\nu) \cap [T(\xi) \circ (\id \times h)]) \\
&= \rho_* (u_{BRS} (\nu) \cap T(\xi)_* \sigma [h]) \\
&= \xi^!_{BRS} [h],
\end{align*}
as was to be shown.
\end{proof}
We will refer to the map $\rho_* (u_{BRS} (\nu)\cap -)$
as the \emph{geometric Thom homomorphism}.

\begin{prop} \label{prop.htpythomhomisgeothom}
The homotopy-theoretic Thom homomorphism $\Phi$ agrees with the geometric
Thom homomorphism, that is, the diagram
\[ \xymatrix{
\widetilde{\MWITT}_{n+s} (\Th (\mu)) \ar[r]^\Phi \ar[d]_{\cong}
& \MWITT_{n+d} (X) \ar[d]^{\cong} \\
\widetilde{\Omega}^\Witt_{n+s} (\Th (\nu)) \ar[r]^{\rho_* (u_{BRS} (\nu) \cap -)} 
& \Omega^\Witt_{n+d} (X)
} \]
commutes.
\end{prop}
\begin{proof}
Recall that $\Phi$ is given by
$\Phi = \rho_* (u_\SPL (\mu) \cap -)$.
The result follows from Lemma \ref{lem.usplgoestoubrs}
applied to $\mu$ with underlying oriented block bundle $\mu_\PLB = \nu$,
together with the geometric description of the cap product
given in \cite{buonrs}.
\end{proof}

\begin{prop} \label{prop.transfchiischitransf}
Manifold-block bundle transfer on $\MWITT$-homology 
and geometric pullback transfer on Witt bordism agree,
that is, the diagram
\[ \xymatrix{
\MWITT_n (B) \ar[r]^{\xi^!} \ar[d]^\cong 
  & \MWITT_{n+d} (X) \ar[d]_\cong \\
\Omega^\Witt_n (B) \ar[r]^{\xi^!_\PB} & \Omega^\Witt_{n+d} (X)
} \]
commutes.
\end{prop}
\begin{proof}
We must show that the outer square of the diagram
\[ \xymatrix@R=30pt@C=30pt{
\MWITT_n (B) \ar[rr]^{\xi^!} \ar[rd]^{T(\xi)_* \sigma} \ar[ddd]^{\cong} &
  & \MWITT_{n+d} (X) \ar[ddd]_{\cong} \\
  & \widetilde{\MWITT}_{n+s} (\Th (\mu)) \ar[d]^{\cong} 
                 \ar[ru]^{\Phi} & \\
  & \widetilde{\Omega}^\Witt_{n+s} (\Th (\nu)) 
        \ar[rd]^{\rho_* (u_{BRS} (\nu) \cap -)} & \\
\Omega^\Witt_n (B) \ar[rr]^{\xi^!_\PB} \ar[ru]^{T(\xi)_* \sigma} & & 
  \Omega^\Witt_{n+d} (X) 
} \]
commutes. The upper part commutes by definition of the $F$-block bundle
transfer $\xi^!$. The left hand part commutes as the vertical arrows are
given by a natural isomorphism of homology theories, while the right hand part
commutes by Proposition \ref{prop.htpythomhomisgeothom}.
The lower part of the diagram, involving the pullback transfer $\xi^!_\PB$,
commutes according to Proposition \ref{prop.geopullbtransfistpcapu}.
\end{proof}

A closed $n$-dimensional Witt space $W$ has a \emph{fundamental class}
\[ [W]_\Witt \in \Omega^\Witt_n (W) \]
in Witt bordism represented by the identity map,
$[W]_\Witt = [\id: W\to W]$.
Under the natural identification (\ref{equ.mwittisomegawitt}),
this class corresponds to a unique class
$[W]_\Witt \in \MWITT_n (W).$

\begin{prop} \label{prop.transferofwittfundclass}
Suppose $B$ is a closed Witt space of dimension $n$. 
Then the total space $X$ of the oriented $F$-block bundle $\xi$
over $B$ is a closed Witt space and
the geometric pullback transfer 
\[ \xi^!_\PB:  \Omega^\Witt_n (B) \longrightarrow
         \Omega^\Witt_{n+d} (X) \]
maps the Witt fundamental class of $B$ to 
the Witt fundamental class of $X$,
\[ \xi^!_\PB [B]_\Witt = [X]_\Witt. \]         
\end{prop}
\begin{proof}
If the base $B$ is Witt, then the total space $X$ is Witt
by Lemma \ref{lem.fiberprodiswitt}.
The Witt fundamental class $[B]_\Witt$ is represented by
the identity map $g=\id_B: B\to B$ (which is PL).
Pulling back under this identity map, 
the map $\overline{g}: E(\id^* \xi) \to X$
is the identity $\id: E(\id^* \xi) =X \to X$.
Therefore,
\[ \xi^!_\PB [\id: B\to B] = [\overline{g}: E(\id^* \xi) \to X] 
  = [\id_X] = [X]_\Witt. \]
\end{proof}

\begin{example} \label{exple.trivialxitransferwittbord}
We continue our previous examples on
the trivial $F$-block bundle $\xi$
with total space $X=F \times B$, $B$ any compact polyhedron.
The geometric pullback transfer 
$\xi^!_\PB:  \Omega^\Witt_n (B) \to \Omega^\Witt_{n+d} (F \times B)$
is then by construction 
$\xi^!_\PB [g:W\to B] = [\id_F \times g: F \times W \to F \times B].$
The Witt bordism $\times$-product
\[ \times: \Omega^\Witt_d (F) \times \Omega^\Witt_n (B)
      \longrightarrow \Omega^\Witt_{d+n} (F \times B),~ [h] \times [g] = [h\times g], \]
can be used to decompose the class $[\id_F \times g]$ as
$[F]_\Witt \times [g]$. We thus find that
\[ \xi^!_\PB [g] = [F]_\Witt \times [g], \]
which agrees with (\ref{equ.transfertrivialxiisproduct}). 
If $B=W$ is an $n$-dimensional closed Witt space and
$g$ the identity, then 
$\xi^!_\PB [B]_\Witt = [F]_\Witt \times [B]_\Witt = [F\times B]_\Witt$, 
in agreement with Proposition \ref{prop.transferofwittfundclass}.
\end{example}

\section{Transfer of the $\syml$-Homology Fundamental Class}

In \cite{blm}, we constructed a canonical 
\emph{$\syml (\rat)$-homology fundamental class}
\[ [X]_\mbl \in \syml (\rat)_n (X) \]
for closed $n$-dimensional Witt spaces $X$
using the morphism $\tau:\MWITT \to \syml (\rat)$ of ring spectra.   
This class is the image of the Witt theory
fundamental class $[X]_{\Witt}$ under
the map
\[ \tau_*: \Omega^\Witt_n (X) \cong \MWITT_n (X) \longrightarrow
   \syml (\rat)_n (X), \]
i.e. $[X]_\mbl = \tau_* [X]_{\Witt}.$

\begin{thm} \label{prop.gysinpreserveslhomfundclass}
Suppose $B$ is a closed Witt space of dimension $n$. 
Then the total space $X$ of the oriented $F$-block bundle $\xi$ over $B$
is a closed Witt space and the block bundle transfer 
\[ \xi^!:  \syml (\rat)_n (B) \longrightarrow
         \syml (\rat)_{n+d} (X) \]
maps the $\syml (\rat)$-homology fundamental class of $B$ to the
$\syml (\rat)$-homology fundamental class of $X$,
\[ \xi^! [B]_\mathbb{L} = [X]_\mathbb{L}. \]         
\end{thm}
\begin{proof}
By Proposition \ref{prop.transferofwittfundclass},
$\xi^!_\PB [B]_\Witt = [X]_\Witt$
for the pullback transfer.
Thus, using Proposition \ref{prop.transfchiischitransf} 
on the compatibility of
block bundle transfer and pullback transfer,
\[ \xi^! [B]_\Witt 
    = \xi^!_\PB [B]_\Witt = [X]_\Witt. \]
Finally, by Proposition \ref{prop.transftaucomm},
\[ \xi^! [B]_\mbl = \xi^! \tau_* [B]_\Witt =
  \tau_* \xi^! [B]_\Witt = \tau_* [X]_\Witt
   = [X]_\mbl. 
\]
\end{proof}

\begin{example} \label{exple.trivialxitransferlhomology}
We describe the $\syml (\rat)$-homology transfer and illustrate
Theorem \ref{prop.gysinpreserveslhomfundclass} for
the trivial $F$-block bundle $\xi$ with total space $X=F \times B$.
We use the notation of the earlier examples on this special case.
By Lemma \ref{lem.umsotoumbl},
$u_\mathbb{L} (\mu_F) = \tau \phi_W (u_\SPL (\mu_F))$.
Hence, using \cite[p. 552, Prop. 7.1.2]{ranickiesitats},
\[ [F]_\mathbb{L} = \rho_{F*} 
   (\tau \phi_W (u_\SPL (\mu_F)) \cap [\Th \mu_F]_\mathbb{L}) 
     \in \syml (\rat)_d (F), \]
see also \cite[p. 186, Prop. 16.16 (c)]{ranickialtm}.
Consequently, Formula (\ref{equ.transfertrivialxiisproduct}) 
applies to yield the description
\[ \xi^! (a) = [F]_\mathbb{L} \times a \]
for the transfer $\xi^!: \syml (\rat)_n (B) \to \syml (\rat)_{n+d} (F\times B)$.
When $B$ is a closed $n$-dimensional Witt space,
we obtain
\[ \xi^! [B]_\mathbb{L} = [F]_\mathbb{L} \times [B]_\mathbb{L}
  = [F\times B]_\mathbb{L} \]
(where the second equality has been established in
\cite[Theorem 13.1]{blm}), in agreement with 
Theorem \ref{prop.gysinpreserveslhomfundclass}.
\end{example}

\section{Behavior of the Cheeger-Goresky-MacPherson L-class Under Transfer}

Rationally, Theorem \ref{prop.gysinpreserveslhomfundclass}
leads to a formula that describes the behavior of the
Cheeger-Goresky-MacPherson $L$-class under block bundle transfer.

\begin{thm}  \label{thm.lclassbundletransfer}
Let $B$ be a closed Witt space
and let $F$ be a closed oriented PL manifold.
Let $\xi$ be an oriented PL $F$-block bundle over $B$
with total space $X$ and
oriented stable vertical normal microbundle $\mu$ over $X$.
Then the associated block bundle transfer $\xi^!$ sends the
Cheeger-Goresky-MacPherson $L$-class of $B$ to the product
\[ \xi^! L_* (B) = L^* (\mu) \cap L_* (X). \]
\end{thm}
\begin{proof}
By Theorem \ref{prop.gysinpreserveslhomfundclass},
the $\syml$-homology transfer $\xi^!$ of $\xi$
sends the $\syml (\rat)$-homology fundamental class of $B$ to the 
$\syml (\rat)$-homology fundamental class of $X$:
$\xi^! [B]_\mbl = [X]_\mbl.$
It remains to analyze what this equation means after we tensor with
$\rat$, i.e. after we apply the localization morphism
\[
\syml (\rat) \longrightarrow \syml (\rat)_{(0)}
 = \bigvee_i S^i H(L^i (\rat)\otimes \rat)
 = \bigvee_j S^{4j} H\rat,
\]
which is a ring morphism of ring spectra.
By \cite[Lemma 11.1]{blm},
\[ [B]_\mbl \otimes \rat = L_* (B),~
   [X]_\mbl \otimes \rat = L_* (X). \]
Using Ranicki's \cite[Remark 16.2, p. 176]{ranickialtm},
\[ u_\mbl (\mu) \otimes \rat 
  = \rho^* L^* (\mu)^{-1} \cup u_\rat (\mu), \]
where $u_\rat (\mu) \in \redh^{s-d} (\Th (\mu);\rat)$ is the Thom class of $\mu$
in ordinary rational cohomology.   
(Note that Ranicki omits cupping with $u_\rat (\mu)$ in his notation.)
Thus
\begin{align*}
L_* (X)
&= [X]_\mbl \otimes \rat 
    = (\xi^! [B]_\mbl) \otimes \rat \\
&= \rho_* (u_\mbl (\mu) \cap T(\xi)_* \sigma [B]_\mbl) \otimes \rat \\
&= \rho_* (u_\mbl (\mu)\otimes \rat \cap T(\xi)_* \sigma ([B]_\mbl \otimes \rat)) \\
&= \rho_* ((\rho^* L^* (\mu)^{-1} \cup u_\rat (\mu)) \cap T(\xi)_* \sigma L_* (B)) \\
&= \rho_* (\rho^* L^* (\mu)^{-1} \cap (u_\rat (\mu) \cap T(\xi)_* \sigma L_* (B))) \\
&= L^* (\mu)^{-1} \cap \rho_* (u_\rat (\mu) \cap T(\xi)_* \sigma L_* (B)) \\
&= L^* (\mu)^{-1} \cap \xi^! L_* (B).
\end{align*}
\end{proof}
If $t$ is a stable inverse for $\mu$, then $t$ has the interpretation
of a stable vertical tangent bundle for $\xi$, and 
by Theorem \ref{thm.lclassbundletransfer}, the formula
\[ \xi^! L_* (B) = L^* (t)^{-1} \cap L_* (X) \]
holds.

\begin{example}
We discuss Theorem \ref{thm.lclassbundletransfer}
vis-\`a-vis Formula (\ref{equ.transfertrivialxiisproduct}) 
in the situation of a trivial $F$-block bundle $\xi$ over $B$, using
the notation of earlier examples on this case.
Let $[F]_\rat \in H_d (F;\rat)$ denote the fundamental class of the oriented
PL manifold $F$ in ordinary rational homology.
By (\ref{equ.transfertrivialxiisproduct}),
\[ \xi^! (a) = [F]_\rat \times a \]
for $a\in H_n (B;\rat)$.
For a closed Witt space $B$, we obtain in particular 
\begin{equation} \label{equ.transflbislbtimesfrat}
\xi^! L_* (B) = [F]_\rat \times L_* (B). 
\end{equation}
Let $TF$ denote the tangent PL microbundle of the PL manifold $F$.
Then $\mu_F \oplus TF$ is the trivial microbundle and hence
$L^* (\mu_F) L^* (TF) = L^* (\mu_F \oplus TF) = 1$.
Furthermore, the Hirzebruch signature theorem holds for PL manifolds
and $L_* (F) = L^* (TF) \cap [F]_\rat$
(see Madsen-Milgram \cite[Chapter 4C]{madsenmilgram} and 
Thom \cite{thom}).
According to Theorem \ref{thm.lclassbundletransfer},
\begin{align*} 
\xi^! L_* (B) 
&= L^* (\mu) \cap L_* (X) \\
&= (L^* (\mu_F) \times 1) \cap (L_* (F) \times L_* (B)) \\ 
&= (L^* (\mu_F) \cap L_* (F)) \times (1 \cap L_* (B)) \\
&= (L^* (\mu_F) \cap L^* (TF) \cap [F]_\rat) \times L_* (B) \\
&= [F]_\rat \times L_* (B), \\
\end{align*}
confirming (\ref{equ.transflbislbtimesfrat}).
It is perhaps worthwhile to emphasize that transfer does not in general commute with
localization of spectra: 
If $\xi^!_\rat$ denotes the transfer on ordinary rational homology and
$\xi^!_\mbl$ the transfer on $\syml (\rat)$-homology, then
generally $\xi^!_\rat (- \otimes \rat) \not=
  \xi^!_\mbl (-) \otimes \rat$. For example,
\[ \xi^!_\rat ([B]_\mbl \otimes \rat) = \xi^!_\rat (L_* (B)) 
  = [F]_\rat \times L_* (B), \] 
which contains less information than
\[ (\xi^!_\mbl [B]_\mbl) \otimes \rat =
  [F\times B]_\mbl \otimes \rat = L_* (F\times B) = L_* (F) \times L_* (B). \]
\end{example}

\section{Normally Nonsingular Maps}
\label{sec.normnonsingmaps}

Let $f:Y\to X$ be a PL map of closed Witt spaces which is the
composition
\[ \xymatrix{
Y \ar@{^{(}->}[r]^g \ar[rd]_f& Z \ar[d]^p \\
& X
} \]
of an oriented normally nonsingular inclusion $g$ with normal bundle $\nu_g$
followed by the projection $p$ of an oriented PL $F$-fiber bundle $\xi$
with closed PL manifold fiber $F$ and stable vertical normal bundle $\nu_\xi$.
Then $f$ is a \emph{normally nonsingular map} in the sense
of \cite[Def. 5.4.3]{gmih2}. Let $c$ be the codimension of $g$ and
$d$ the dimension of $F$.
The bundle transfer $\xi^!$ and the Gysin restriction $g^!$ compose to
give a transfer homomorphism
\[ H_n (X;\rat) \stackrel{\xi^!}{\longrightarrow}
 H_{n+d} (Z;\rat) \stackrel{g^!}{\longrightarrow}
 H_{n+d-c} (Y;\rat), \]
with $c-d$ the \emph{relative dimension} of $f$.
Combining Theorem \ref{thm.lclassbundletransfer} of the present paper with
Theorem 3.18 of \cite{banaglnyjm}, we obtain
\begin{align*}
g^! \xi^! L_* (X)
&= g^! (L^* (\nu_\xi) \cap L_* (Z)) = g^* L^* (\nu_\xi) \cap g^! L_* (Z) \\
&= g^* L^* (\nu_\xi) \cap (L^* (\nu_g) \cap L_* (Y)) \\
&= L^* (g^* \nu_\xi \oplus \nu_g) \cap L_* (Y),
\end{align*}
at least when $Y,Z$ have even dimensions.

\end{document}